\documentclass{article}
\usepackage{hyperref}
\usepackage{amsmath}
\usepackage{amsthm}
\usepackage{amssymb}
\usepackage{tikz,pgfplots}
\usepackage{tikz-cd}
\usepackage{parselines}
\usepackage{booktabs}
\usepackage{mathrsfs}
\usepackage{setspace}
\usetikzlibrary{arrows.meta}

\usepackage[backend=biber,style=alphabetic,maxnames=5]{biblatex}
\newtheorem{theorem}{Theorem}[section]
\newtheorem{prop}[theorem]{Proposition}
\newtheorem{lemma}[theorem]{Lemma}
\newtheorem{sublemma}[theorem]{Sublemma}
\newtheorem{definition}[theorem]{Definition}

\newtheorem{remark}[theorem]{Remark}

\newtheorem{def-thm}[theorem]{Definition-Theorem}

\numberwithin{equation}{section}
\begin{document}

\title{Extension of Period Maps by Polyhedral Fans}

\author{Haohua Deng}
\date{}

\maketitle              % typeset the title of the contribution

% You don't need an abstract or keywords for an article review
%\begin{abstract}
%The abstract should summarize the contents of the paper
%using at least 70 and at most 150 words. It will be set in 9-point
%font size and be inset 1.0 cm from the right and left margins.
%There will be two blank lines before and after the Abstract. \dots
%\keywords{List up to three keywords here, like this:
%computational geometry, graph theory, Hamilton cycles}
%\end{abstract}

% TO MAKE A TITLE PAGE USE THE FOLLOWING COMMAND HERE.
% \newpage

\begin{center}
\textbf{Abstract}
\end{center}\footnotesize
Kato and Usui developed a theory of partial compactifications for quotients of period domains $D$ by arithmetic groups $\Gamma$, in an attempt to generalize the toroidal compactifications of Ash-Mumford-Rapoport-Tai to non-classical cases.  Their partial compactifications, which aim to fully compactify the images of period maps, rely on a choice of fan which is strongly compatible with $\Gamma$.  In particular, they conjectured the existence of a \emph{complete} fan, which would serve to simultaneously compactify all period maps of a given type.

In this article, we briefly review the theory, and construct a fan which compactifies the image of a period map arising from a particular two-parameter family of Calabi-Yau threefolds studied by Hosono and Takagi, with Hodge numbers $(1,2,2,1)$.  On the other hand, we disprove the existence of complete fans in some general cases, including the $(1,2,2,1)$ case.
\tableofcontents
\normalsize

\section{Introduction}

The boundary behavior of period maps and compactifications of period domains have been studied for decades, and studies on the classical cases are fruitful.  One of the most significant results, provided by Baily and Borel in 1966, is that for a Hermitian symmetric domain $D$ and an arithmetic subgroup $\Gamma\leq Hol(D)^+$, the space $\Gamma\backslash D$ admits a compactification by adding rational boundary components (or ``cusps'') \cite{BB66}. The resulting \emph{Baily-Borel Compactification} $\Gamma\backslash D^*$ is a projective variety, usually highly singular.  To improve the properties of this compactification, Mumford et al \cite{AMRT10} introduced the \emph{toroidal compactification}, which depends on a choice of $\Gamma$-admissable polyhedral fan $\Sigma$. When putting certain restrictions on $\Gamma$ and $\Sigma$, the space produced by toroidal compactification is as good as a smooth projective variety, which is essentially a blow-up of the Baily-Borel compactification at cusps.

Many attempts to describe boundary structures of general period domains and their subvarieties coming from geometric variations, as well as generalize known compactification theories for classical cases to non-classical cases have been made. One of the ideas is that for a variation of Hodge structure coming from a geometric family with singular fibers over discriminant locus, looking at how the Hodge structures ``degenerate'' locally around the singular fibers could be useful. For such a family, Deligne observed that the holomorphic vector bundle fibered by cohomology of smooth projective varieties can be extended holomorphically across a singular locus of simple-normal-crossing type, see \cite{Del97}.  Moreover, Steenbrink showed in a series of works \cite{Ste76,Ste77,SS85} that such degenerating behaviors can be understood by studying the local nearby and vanishing cycles.  In \cite{Sch73}, Schmid proved the fundamental result that the degeneration of Hodge structures, and the resulting \emph{limiting mixed Hodge structures}, have essential relations with the unipotent monodromy of the family, as described by the nilpotent and $\mathrm{SL}_2$-orbit theorems. 

Recently, Kato and Usui developed their own theory of partial compactifications for arithmetic quotients of general period domains in \cite{KU08}. The basic idea is to define boundary components of period domains as parameter spaces for multi-variable nilpotent orbits.  By putting proper topological structures on the resulting spaces (to ``glue in'' these boundary components), one can provide compactifications whose global structure is encoded by some polyhedral fan of nilpotent elements. This purports to generalize the construction of toroidal compactification in \cite{AMRT10} to non-classical cases. 

More precisely, for a period domain $D$ and an arithmetic group $\Gamma\leq Aut(D)$, Kato and Usui's theory requires a choice of a polyhedral fan $\Sigma\subset \mathfrak{g}_\mathbb{Q}:=End(D)_{\mathbb{Q}}$ which is made of nilpotent elements. The faces of $\Sigma$, called nilpotent cones, are expected to contain information about degeneration of Hodge structures along the boundary of $D$ in its compact dual $\check{D}$ (in the sense of naive limits \cite{KPR19}). It has been proved in \cite{KU08} that if we require the chosen fan $\Sigma$ as well as $\Gamma$ to satisfy some certain conditions, the space $X_{\Sigma}:=\Gamma\backslash D_{\Sigma}$ will be a good compactification of $\Gamma\leq Aut(D)$ under the sense it admits a Hausdorff log analytic structure (see \cite[Chap. 4]{KU08}). Kato and Usui also made conjectures about the complete fan and complete weak fan which, if proved to be true, will guarantee compactification for the image of any period map into $\Gamma\backslash D$, as such fan or weak fan is, intuitively speaking, the collection of all $\Gamma$-compatible nilpotent cones. The conjecture about the existence of complete fan has been verified by Kato and Usui themselves (based on \cite{AMRT10}) for classical cases. However, for non-classical cases, Watanabe provided a counterexample in \cite{Wat08}. The conjecture about the existence of complete weak fan still remains open.

\medskip

In this paper, we are going to briefly outline the literature introduced above, and proceed to Kato and Usui's theory 
in the context of a geometric family studied by Hosono and Takagi in their works \cite{HT14} and \cite{HT18}. In Section \ref{S2}, basic background and some classical results from Hodge theory will be reviewed. Section \ref{S3} will be devoted to a brief review of Kato and Usui's work in \cite{KU08}. Some alternative expository resources could be \cite{GGK13} and \cite{Gr09}. In Section \ref{S4}, we are going to take a closer look at Kato and Usui's conjecture about complete fan. We also introduce our first main result (\textbf{Theorem \ref{thm4.6}}), which generalizes Watanabe's result by a new method:\\

\noindent \textbf{Main Result 1: } \textit{For period domains of Hodge type $(a,b,a)$, $b\geq a\geq 2$, and $(1,m,m,1)$, $m\geq 2$, there is no complete fan.}\\

Though it is unlikely we will have a complete fan for any non-classical cases, we can still attempt to construct a polyhedral fan which contains boundary information sufficient to compactify a given variation of Hodge structure. In Sections \ref{S5}-\ref{S6}, based on Hosono and Takagi's investigation into a family of Calabi-Yau 3-folds of Hodge type $(1,2,2,1)$, we will indeed construct a $\Gamma$-strongly compatible polyhedral fan for this family.  Note that this VHS is Mumford-Tate generic, meaning that the image of the period map does not factor through a Mumford-Tate subdomain of the period domain (cf. \textbf{Theorem \ref{thm5.4}}).  In particular, the period map is ``strictly nonclassical'' and a compactification cannot follow from \cite{AMRT10}.  Our second main result (\textbf{Theorem \ref{thm5.3}}) can be summarized as follows:\\

\noindent \textbf{Main Result 2: } \textit{There is an arithmetic subgroup $\Gamma\leq Sp(6,\mathbb{Z})$ of finite index and a fan $\Sigma$ in $\mathfrak{sp}_6$ strongly compatible with $\Gamma$, such that the Kato-Usui space $\Gamma\backslash D_{\Sigma}$ compactifies the image of Hosono and Takagi's period map.}\\

In the Appendix, we provide two examples with Hodge numbers $(2,2,2)$ and $(1,2,2,1)$ to illustrate our proof in Section \ref{S4}.  More precisely, given fans made of nilpotent elements, we are going to show how to construct nilpotent orbits that do not come from these fans and hence lead to a counterexample for the conjecture of complete fans.

\bigskip

\noindent \textbf{Acknowledgement:} The author is grateful to Prof. Matt Kerr, the author's Ph.D thesis advisor, for his help on this project including introducing the topic, explanation on some critical concepts and careful revision. The author also thanks Prof. Patricio Gallardo for generously sharing his ideas as well as Prof. Shinobu Hosono for clarifying some computational results.

\newpage
\section{Preliminaries}\label{S2}

We begin with some basic definitions and concepts in Hodge theory.
\begin{definition} 
\begin{description}
\item[I. ]
Let $R$ be a $\mathbb{Z}$-algebra, usually one of $\mathbb{Z}, \mathbb{Q}, \mathbb{R}$, and $H_{R}$ be a free $R$-module. A \textbf{$R$-Hodge structure} of weight $n$ and type $\{h^{p,q}\}$ where $p,q\geq 0, p+q=n$ acquires the following equivalent definitions:
 \begin{description}
  \item[(i). ] A decomposition of $H_\mathbb{C} := H_R \bigotimes \mathbb{C}$: 
    \begin{center}
    $H_\mathbb{C} = \bigoplus_{p+q=n}H^{p,q} $ with $\overline{H^{p,q}}=H^{q,p}$ and $dim(H^{p,q})=h^{p,q}$.
    \end{center}
  \item[(ii). ] A decreasing filtration $H_\mathbb{C} = F_0 \supseteq F_1 \supseteq ... \supseteq F_n\supseteq \{0\}$ such that $H_{\mathbb{C}}=F_p\bigoplus \overline{F_{n+1-p}}$ for $0\leq p\leq n$ and $dim(F^p\cap \overline{F^q})=h^{p,q}$ for $p+q=n$.
  It is related with (i) by setting $F^k:=\bigoplus_{p\geq k}H^{p,q}$ and $H^{p,q}:= F^p\cap \overline{F^q}$.
  \item[(iii). ] A morphism of algebraic groups $\varphi: U(1)\rightarrow SL(H)$ such that $\varphi(-1)=(-1)^nId_H$.  It is related with (i) by letting $H^{p,q}$ be the $z^p\bar{z}^q$-eigenspace of $\varphi(z)\in SL(V_{\mathbb{C}})$.
 \end{description}
\item[II. ] For two Hodge structures $\{H^{p,q}\}, \{H'^{p,q}\}$, both of weight $n$, a \textbf{morphism of $R$-Hodge structures} is an $R$-linear map from $H$ to $H'$ whose complexification sends $H^{p,q}$ to $H'^{p,q}$ ($\forall p,q$).
\item[III. ] A \textbf{polarized Hodge structure} of weight $n$ consists of a Hodge structure of weight $n$, defined by data $(H_\mathbb{Z}, \{H^{p,q}\}, \{F^p\})$ together with a non-degenerate bilinear form $Q: H_{\mathbb{Z}}\times H_{\mathbb{Z}} \rightarrow \mathbb{Z}$ such that:
   \begin{description}
     \item[(i). ] $Q$ is symmetric (resp. skew-symmetric) if $n$ is even (resp. odd).
     \item[(ii). ] $Q(x,y)=0$ if $x\in H^{p,n-p}, y\in H^{q,n-q}, p+q\neq n$.
     \item[(iii). ] For any $x \neq 0$ in $H^{p,q}$, $i^{p-q}Q(x,\bar{x})>0$.
   \end{description}
\end{description}
\end{definition}

\begin{remark}
When $R=\mathbb{Q}$ or $\mathbb{R}$, let $G_R:=Aut(H,R)$ and its Lie algebra $\mathfrak{g}_R:=End(H,R)$. A Hodge structure of weight $n$ on $H_R$ induces a Hodge structure of weight $0$ on $End(H,R)$, by $End(H,\mathbb{C})^{k,-k}:=\{\phi \in End(H,\mathbb{C})| \phi(H^{p,q})\subset H^{p+k,q-k}, \forall  p+q=n\}$. If the Hodge Structure on $H_R$ is polarized by non-degenerate bilinear form $Q$ on $H_R$, we replace $End(H,R)$ by $End_{Q}(H,R)$.
\end{remark}

\begin{definition}
Given $H_\mathbb{Z}, n\in \mathbb{Z}_{>0}, h=(h^{p,q})_{p+q=n}, f^p=\sum_{k\geq p}h^{k,n-k}, Q: H_{\mathbb{Z}}\times H_{\mathbb{Z}} \rightarrow \mathbb{Z}$ satisties conditions above, denote by $D$ the \textbf{classifying space of polarized Hodge structure} (or \textbf{period domain}) of the given type, which consists of all weight-$n$ polarized Hodge Structures on $H_{\mathbb{R}}$ with Hodge numbers $h^{p,q}$ and polarized by $Q$. We may express $D$ as a subset of a product of Grassmanians
\begin{equation}
    D=\{(F^0, F^1, ..., F^n) \in \Pi_{p=0}^{n} Grass(f^p, dim_{\mathbb{C}}(H_{\mathbb{C}}))\mid \textup{\textbf{(i)}-\textbf{(v)}}\text{ hold}\} 
\end{equation}
where
\begin{description}
   \item[(i). ] $F^{p-1} \supset F^p$ for $1\leq p\leq n$.
   \item[(ii). ] $dim_{\mathbb{C}}(F^p/F^{p+1})=h^{p,q}$ for $0\leq p\leq n-1$.
   \item[(iii). ] $F^p\bigoplus \overline{F^q}=H_{\mathbb{C}}$ for $p+q=n$.
   \item[(iv). ] $Q(F^p, F^q)=0$ when $p+q>n$.
   \item[(v). ] For $x\in F^p\cap \overline{F^q}$ and $x\neq 0$, we have $i^{p-q}Q(x,\bar{x})>0$.
\end{description}
We also define the space $\check{D}$, called the \textbf{compact dual} of $D$, to be the algebraic subvariety of $\Pi_{p=0}^{n} Grass(f^p, dim_{\mathbb{C}}(H_{\mathbb{C}}))$ given by conditions \textup{\textbf{(i)}}, \textup{\textbf{(ii)}}, and \textup{\textbf{(iv)}}.  
\end{definition}

\begin{remark}
\begin{description}

\item[1. ]$D$ is an analytic open subset of $\check{D}$. For example, when $n=1$ and $dim_{\mathbb{C}}(H_{\mathbb{C}})=2$, let $\{e_1, e_2\}$ be a basis of $H_{\mathbb{C}}$ and $Q$ be the skew-symmetric bilinear form defined by $Q(e_1, e_2)=-1$, then $\check{D}\cong \mathbb{P}^1$ and $D\cong \{v \in \mathbb{C}| \mathfrak{Im}(v)>0\}$. 
\item[2. ] Fix a reference point $F^{\bullet} \in D$. It is clear $G_{\mathbb{R}} := Aut(H_{\mathbb{R}}, Q)$ acts on $D$ with stabilizer of $F^{\bullet}$ denoted by $V$. This action is transitive and thus we have $D\cong G_{\mathbb{R}}/V$. Similarly we have a $G_{\mathbb{C}} := Aut(H_{\mathbb{C}}, Q)$-action on $\check{D}$ with stabilizer of $F^{\bullet} \in \check{D}$ denoted as $P$, so that $\check{D}\cong G_{\mathbb{C}}/P$.
\item[3. ] Again fixing a reference point $F^{\bullet} \in \check{D}$, the tangent space $T_{F^{\bullet}}\check{D}$ can be identified with a subspace of $\bigoplus_{0\leq p\leq n}Hom(F^p, H_{\mathbb{C}}/F^p)$; see for example, \cite[Chap. 7]{CZGT13}. We define the \textbf{horizontal tangent space} $T^h_{F^{\bullet}}\check{D}$ to be its intersection with the subspace $\bigoplus_{0\leq p\leq n}Hom(F^p, F^{p-1}/F^p)$. It is clear $T_{F^{\bullet}}D=T_{F^{\bullet}}\check{D}$. Any submanifold of $D$ whose tangent bundle is a subbundle of the horizontal tangent bundle is called a \textbf{variation of Hodge structure}.
\end{description}
\end{remark}

The structure of boundary components of $D$ in $\check{D}$ can be studied by nilpotent orbits and limiting mixed Hodge structures. To start with, we introduce the concept of (polarized) mixed Hodge structures.

\begin{def-thm}
Suppose we have a period domain $D$ parametrizing polarized Hodge structures of type $(h^{p,q}_{p+q=n}, H_{\mathbb{Z}}, Q)$, with compact dual $\check{D}$.  Let $F^{\bullet}\in \check{D}$ and $W_{\bullet}$ be an increasing filtration on $H_{\mathbb{Q}}$. 
\begin{description}
\item[(I). ] We say the pair $(W_{\bullet}, F^{\bullet})$ gives a \textbf{mixed Hodge structure}, if for any $l \in \mathbb{Z}$ and $Gr_{l}^{W_{\bullet}}:=W_{l,\mathbb{C}}/W_{l-1, \mathbb{C}}$, $F^{\bullet}$ induces a weight $l$ Hodge structure on $Gr_{l}^{W_{\bullet}}$. Here the induced filtration by $F^{\bullet}$ is given by $F^pGr_{l}^{W_{\bullet}}:=F^pW_{l,\mathbb{C}}/F^pW_{l-1, \mathbb{C}}$.
\item[(II). ] Let $N$ be a (rational) nilpotent element in $\mathfrak{g}_{\mathbb{Q}}=End_{\mathbb{Q}}(H, Q)$. There exists a unique increasing filtration $W(N)_{\bullet}$ which is called the associated \textbf{weight filtration} or \textbf{Jacobson-Morozov filtration} on $H_{\mathbb{Q}}$, such that:
  \begin{description}
    \item[(i). ] $0=W(N)_{-m-1}\subset W(N)_{-m}\subset ... \subset W(N)_{m}=H_{\mathbb{Q}}$ where $m$ is defined by $N^m\neq 0, N^{m+1}=0$.
    \item[(ii). ] $N(W_l(N))\subset W_{l-2}(N)$.
    \item[(iii). ] $N^l: Gr_{l}^{W(N)}\rightarrow Gr_{-l}^{W(N)}$ is an isomorphism.
  \end{description}
\item[(III). ] Use the notations above, we say $(W_{\bullet}, F^{\bullet})$ is a \textbf{$N$-polarized mixed Hodge structure} on $(H,Q)$ if:
  \begin{description}
    \item[(i). ] $W_{\bullet} = W(N)_{\bullet}[-n]$.
    \item[(ii). ] $N(F^p)\subset F^{p-1}$.
    \item[(iii). ] Writing $P_{n+l}:=ker\{N^{l+1}: Gr_{n+l}\rightarrow Gr_{n-l-2}\}$ and the non-degenerate bilinear form $Q_l(v,w):=Q(v,N^lw)$ on $P_{n+l}$, $F^{\bullet}$ induces a weight $n+l$ Hodge structure on $P_{n+l}$ polarized by $Q_l$.
  \end{description}
\end{description}
\end{def-thm}

\begin{definition}\label{def2.6}
Let $F^{\bullet}\in \check{D}$ and $\{N_j\}_{1\leq j\leq k}$ be nilpotent elements in $\mathfrak{g}_{\mathbb{Q}}$. Let $\sigma$ be the rational polyhedral cone generated by $\{N_j\}_{1\leq j\leq k}$. We call $\sigma$ a \textbf{nilpotent cone} and $Z:=exp(\sigma_{\mathbb{C}})F^{\bullet} \subset \check{D}$ a $\sigma$-\textbf{nilpotent orbit} if the following conditions are satisfied:
\begin{description}
  \item[(i) (Commutativity).] $[N_i, N_j]=0$ for $\forall i,j$.
  \item[(ii) (Griffiths Transversality).] $N_i(F^p)\subset F^{p-1}$ for $\forall i,p$.
  \item[(iii) (Positivity).] For $\{z_j\}_{1\leq j\leq k}\subset \mathbb{C}$, and $\mathfrak{Im}(z_j)>>0$, we have $$exp(\textstyle\sum_{1\leq j\leq k}z_jN_j)F^{\bullet}\in D.$$
\end{description}
\end{definition}

\begin{remark}
It is clear that $F^{\bullet}$ can be replaced by $exp(\sum_{1\leq j\leq k}\alpha_jN_j)F^{\bullet}$ for any $(\alpha_1,...,\alpha_k)\in \mathbb{C}^k$ without changing the nilpotent orbit. The action of changing the base point in this way is called \textbf{rescaling}.
\end{remark}

For any nilpotent element $N\in int(\sigma)$, consider its associated weight filtration $W(N)_{\bullet}$. Cattani and Kaplan proved the following significant result, which offers a practical way to characterize nilpotent orbits.

\begin{theorem}\label{thm2.8}
\textbf{\cite[Thm. 3.3]{CK82}} Suppose $\sigma$ is a rational polyhedral cone underlying a nilpotent orbit. For each face $\tau\leq \sigma$, there exists an increasing filtration $W(\tau)_{\bullet}$ such that $W(N)_{\bullet}=W(\tau)_{\bullet}$ whenever $N\in int(\tau)$.
\end{theorem}

The following definition is motivated by the purpose of describing nilpotent orbits by their associated monodromy weight filtrations:

\begin{def-thm}
For a nilpotent orbit $Z:=exp(\sigma_{\mathbb{C}})F^{\bullet} \subset \check{D}$, let $W_{\bullet}:=W(\sigma)_{\bullet}$ be the weight filtration on $H_{\mathbb{Q}}$ generated by any element $N\in int(\sigma_{\mathbb{Q}})$. The pair $(W_{\bullet}, F^{\bullet})$ gives a mixed Hodge structure on $H_{\mathbb{Q}}$, which is called the \textbf{limiting mixed Hodge structure} associated to $Z$.
\end{def-thm}

We turn to submanifolds of a period domain $D$ coming from geometric families. Let $B$ be a connected smooth quasi-projective variety. Suppose $\bar{B}$ is a smooth projective variety and $\bar{B}\backslash B$ consists of simple normal crossing divisors. Consider a quasiprojective manifold $\mathcal{X}$ together with a holomorphic submersion $\phi: \mathcal{X}\rightarrow B$. For $b\in B$, denote the fiber $X_b:=\phi^{-1}(b)$. Assume that for any $b\in B$, $X_b$ is a smooth projective variety.

Consider the local system $R^n\phi_*(\underline{\mathbb{Z}}_{\mathcal{X}})$ on $B$. The fiber over $b\in B$ is exactly $H^n(X_b, \mathbb{Z})$; its complexification admits the well-known Hodge decomposition
\begin{equation}
    H^n(X_b, \mathbb{C})=\bigoplus_{p+q=n}H^{p,q}(X_b),
\end{equation}
where $H^{p,q}(X_b)$ is the subspace of de Rham cohomology of $X_b$ represented by $(p,q)$-forms. It is also well-known fact that $h^{p,q}(t):=dim_{\mathbb{C}}H^{p,q}(X_b)$ are constants for $b\in B$ and the bundles $\mathbb{F}^l_b:=\bigoplus_{p\geq l}H^{p,n-p}(X_b)$ are holomorphic subbundles, called the \textbf{Hodge bundles} of the family. Therefore, for fibers $\{X_b\}_{b\in B}$, we have Hodge Structures of weight $n$ given by the holomorphic sections of the Hodge bundles.

Now fix a point $b_0\in B$. For any point $b_1\in B$ and a path $\gamma$ from $b_0$ to $b_1$, there is an induced map $\bar{\gamma}: H^n(X_{b_0}, \mathbb{Z}) \rightarrow H^n(X_{b_1}, \mathbb{Z})$ which depends only on the homotopy class of $\gamma$. In particular, by considering all loops in $B$ based on $b_0$ and their homotopy classes, we have the map
\begin{equation}
    \rho: \pi_1(B,b_0)\rightarrow Aut(H^n(X_b), \mathbb{Z}),
\end{equation}
called the \textbf{monodromy representation} of the family. Denote 
\begin{equation}
\Gamma:=Img(\rho)\leq Aut(H^n(X_b), \mathbb{Z}), 
\end{equation}
and let $D$ be the period domain of the type given by the Hodge numbers. It is clear that we have a natural map
\begin{equation}
    \varphi: B\rightarrow \Gamma \backslash D,
\end{equation}
called the \textbf{period map} associated to the family. In the remainder of this paper, we will focus on studying extensions of $\varphi$ from $B$ to $\bar{B}$, with image in a properly-chosen compactification of $\Gamma\backslash D$.

%Now we move to the concept of variations of Hodge Structures\textbf{VHS}. There are two alternative ways to study the concepts: By its abstract %definitions, or by those objects come from geometric families. We study both of them.

%\begin{definition}
%Let $B$ be a connected complex manifold. An (abstract) \textbf{Variation of Hodge Structure} on $B$ is given by a local system $\mathbb{V}_{\mathbb{Z}}$ on %$B$ with its flat connection $\nabla$, together with a sequence of holomorphic subbundles of $\mathbb{V}_{\mathbb{C}}$: $\mathbb{V}_{\mathbb{C}} = %\mathbb{F}^0 \supset \mathbb{F}^1 \supset \mathbb{F}^2 \supset ... \supset \mathbb{F}^n \supset \mathbb{F}^{n+1}={0}$ such that:
%\begin{description}
% \item[(i). ] $\mathbb{V}_{\mathbb{C}} = \mathbb{F}^p\bigoplus \overline{\mathbb{F}^{n-p+1}}$ as smooth vector bundles.
% \item[(ii). ] $\nabla(\mathcal{F}^p) \subset \Omega_{B}^1\bigotimes \mathcal{F}^{p-1}$, where $\mathcal{F}^p$ is the sheaf of holomorphic sections of %$\mathbb{F}^p$ and $\Omega_{B}^1$ is the sheaf of holomorphic differential $1$=form on $B$.
%\end{description}
%\end{definition}

%Given a Variation of Hodge Structure on a connected complex manifold $B$, on any $b\in B$, the fiber $\mathbb{V}_b$ admits a Hodge Structure given by %$\{\mathbb{F}_b^p\}$. Moreover, it is a fact that $f^p:=dim_{\mathbb{C}}(\mathbb{F}_b^p)$ are constants, so do constants %$h^{p,q}:=dim_{\mathbb{C}}(\mathbb{F}_b^p\cap \overline{\mathbb{F}_b^q})$. However, the bundles $\mathbb{V}^{p,q}:=\mathbb{F}^p\cap %\overline{\mathbb{F}^q}$ are not holomorphic.

Now we consider the period map locally. Suppose $B=(\Delta^*)^k\subset \bar{B}=(\Delta)^k$ where (resp. $\Delta^*$) $\Delta$ is the (resp. punctured) unit disk in $\mathbb{C}$. In other words, the divisor $\mathbb{D}:=\bar{B}\backslash B$ is determined by the equation $z_1z_2...z_k=0$ where $z_i$ are coordinates on $\mathbb{C}^k$. For a fixed point $b_0\in B$, $\pi_1(B,b_0)$ is free abelian and generated by $\gamma_1,...,\gamma_k$, where $\gamma_i$ is represented by a small loop based on $b_0$ and goes around the divisor $\mathbb{D}_i:=\{x_i=0\}$. Let $T_i\in \Gamma$ be the image of $\gamma_i$ under the monodromy representation.

\begin{theorem}
\textbf{(Monodromy Theorem)} All $T_i$ are quasi-unipotent, which means that there exist positive integers $w_i, v_i$ such that $(T_i^{w_i}-I)^{v_i}= 0$ for every $i$. Moreover, $v_i\leq n+1$ for each $i$.
\end{theorem}

Therefore, by performing a finite base change, we can assume all $T_i$ are unipotent. Denote $N_i:=log(T_i)$ as the monodromy logarithms, $\mathfrak{H}:=\{z\in \mathbb{C}| \mathfrak{Im}(z)>0\}$ as the upper-half plane. Consider the universal covering map:
\begin{equation}
\Phi: \mathfrak{H}^k\rightarrow (\Delta^*)^k, \Phi(\vec{z})=exp(2\pi i\vec{z}), \vec{z}=(z_1,...,z_k).
\end{equation}
We have the following commutative diagram:

\begin{equation}
\begin{tikzcd}
\mathfrak{H}^k \arrow[d, "\Phi"] \arrow[r, "\tilde{\varphi}"] & D \arrow[d] \\
B=(\Delta^*)^k \arrow[r, "\varphi"] & \Gamma \backslash D
\end{tikzcd}
\end{equation}

\noindent It is also clear that the action of $T_1,\ldots,T_k$ on $D$ can be interpreted as:
\begin{equation}
   \tilde{\varphi}(z_1,...,z_i+1,...,z_k)=T_i\tilde{\varphi}(z_1,...,z_k), T_i=exp(N_i).
\end{equation}
Therefore, to get a single-valued map from $(\Delta^*)^k$ to $\check{D}$, we need to ``untwist'' the map $\tilde{\varphi}$. Let $\Psi: H^k\rightarrow \check{D}$ and $\psi: (\Delta^*)^k\rightarrow \check{D}$ be defined as:
\begin{equation}
    \Psi(z_1,...,z_k):= exp(\sum_{1\leq i\leq k}-z_iN_i)\tilde{\varphi}(z_1,...,z_k),
\end{equation}
\begin{equation}
    \psi(e^{2\pi iz_1},...,e^{2\pi iz_k}):= \Psi(z_1,...,z_k)
\end{equation}

%It is clear that $\psi: (\Delta^*)^k\rightarrow \check{D}$ is well-defined. 
\noindent Regarding the properties of $\psi$, Schmid proved the following fundamental result.

\begin{theorem}
\textbf{\cite{Sch73} (Schmid's Nilpotent Orbit Theorem)}
\begin{description}
  \item[(i). ] The map $\psi$ can be extended holomorphically to $\Delta^k$.
  \item[(ii). ] With the extension in \textbf{(i)}, the set $Z:=exp(\sum_{1\leq j\leq k}z_jN_j)\psi(0)\subset \check{D}$ is a nilpotent orbit. 
\end{description}
\end{theorem}

At this point, it is clear that for a variation of Hodge structure on a connected complex manifold $B$, with $\mathbb{D}=\bar{B}\backslash B$ a union of simple normal crossing divisors, the behaviors of the period map around strata of $\mathbb{D}$ are characterized by nilpotent orbits and associated Limiting Mixed Hodge Structures.

\section{The theory of Kato-Usui Spaces}\label{S3}
In this section, we briefly review the theory of Kato-Usui spaces \cite{KU08}.  We refer the chapter $0$ of \cite{KU08} or Appendix to chapter $10$ of \cite{GGK13} for an overview, and \cite{Gr09} for an introductory report.

\begin{definition}\label{def3.1}
Let $G$ and $\mathfrak{g}:=Lie(G)$ be as defined before. A \textbf{fan} $\Sigma\subset \mathfrak{g}_{\mathbb{Q}}$ is a collection of sharp rational nilpotent cones in $\mathfrak{g}_{\mathbb{Q}}$ that satisfies the following conditions:
\begin{description}
  \item[(i). ] $\Sigma$ is closed under the action of taking faces.
  \item[(ii). ] For any $\sigma, \sigma' \in \Sigma$, $\sigma \cap \sigma'$ is a face of both of $\sigma, \sigma'$.
\end{description}
Here we say a cone $\sigma$ is \textbf{sharp} if $\sigma \cap (-\sigma)=\{0\}$.
\end{definition}

\begin{remark}\label{rem3.2}
If we replace the condition (ii) in \textbf{Definition \ref{def3.1}} by 
\begin{description}
  \item[(ii)'. ] If $\sigma, \sigma' \in \Sigma$ have a common interior point and there exists $F^{\bullet}\in \check{D}$ such that both of $\exp(\sigma_{\mathbb{C}})F^{\bullet}$ and $\exp(\sigma'_{\mathbb{C}})F^{\bullet}$ are nilpotent orbits, then $\sigma=\sigma'$.
\end{description}
then we call the resulting $\Sigma$  a \textbf{weak fan}.
\end{remark}

A natural approach to extending a period map $B\to \Gamma\backslash D$ to $\bar{B}$ is to add nilpotent orbits. This motivates the following definition: given a fan $\Sigma \subset \mathfrak{g}_{\mathbb{Q}}$, put
\begin{center}
    $D_{\Sigma}:=\{(\sigma, Z_{\sigma})|\sigma\in \Sigma$, $Z_{\sigma}$ is a $\sigma$-nilpotent orbit\}/\{rescalings\}.
\end{center}

\noindent It is clear that $D \subset D_{\Sigma}$ by $F^{\bullet} \mapsto (\{0\}, F^{\bullet})$.  In order to work modulo $\Gamma$, we need some compatibility conditions:
\begin{definition}\label{def3.3}
\textbf{\cite[Sec. 0.4.10]{KU08}} Let $\Sigma$ and $\Gamma$ be defined as above.
\begin{description}
  \item[(I). ] We say $\Gamma$ is \textbf{compatible} with $\Sigma$, if $\Sigma$ is closed under $Ad(\Gamma)$-action.
  \item[(II). ] We say $\Gamma$ is \textbf{strongly compatible} with $\Sigma$, if $\Gamma$ is compatible with $\Sigma$, and for any $\sigma\in \Sigma$, $\sigma$ is generated by $log(\Gamma(\sigma))$, where $\Gamma(\sigma):=\Gamma\cap exp(\sigma)$.
\end{description}
\end{definition}
\begin{remark}
When $\Gamma=G_{\mathbb{Z}}$, \textbf{Definition \ref{def3.3} (II)} exactly means that $\Sigma$ comes from logarithms of rational points of $G$.
\end{remark}
Given a strongly compatible pair $(\Gamma,\Sigma)$, we now set
\begin{center}
    $X_{\Sigma}:=\Gamma\backslash D_{\Sigma}$.
\end{center}
The non-trivial boundary components of $X_{\Sigma}$ are given by $\Gamma_{\sigma}\backslash B(\sigma)$, where $\sigma$ is a (non-trivial) rational nilpotent cone, $B(\sigma)$ is the set of nilpotent orbits polarized by $\sigma$ modulo rescaling, and $\Gamma_{\sigma}\leq \Gamma$ is the stabilizer of $\sigma$. (Regarding the computation on dimension of $\Gamma_{\sigma}\backslash B(\sigma)$, we refer to \cite[Chap. 10]{GGK13}.) The upshot is that we may write
\begin{center}
    $X_{\Sigma}=\Gamma\backslash D_{\Sigma}=D\bigcup_{\sigma\leq \Sigma}(\Gamma_{\sigma}\backslash B(\sigma))$.
\end{center}
In the classical cases (where the tangent bundle and the horizontal tangent bundle of $D$ coincide), this construction recovers the toroidal compactifications of \cite{AMRT10}.  All classical cases are listed in \cite[Sec. 0.4.14]{KU08}.

For non-classical cases, Kato-Usui proved that (in a certain topology) $X_{\Sigma}=\Gamma\backslash D_{\Sigma}$ admits the structure of a log-analytic Hausdorff space, with slits caused by Griffiths transversality at the boundaries. We refer readers to Kato-Usui's original work \cite{KU08} for the full story.

\begin{remark}
\begin{description}
\item[(i). ] Suppose that $\Gamma$ is \textbf{neat}, i.e. no element has non-trivial roots of unity as eigenvalues, and $(\Gamma, \Sigma)$ is a strongly compatible pair.  Then $X_{\Sigma}=\Gamma\backslash D_{\Sigma}$ admits a structure of \textbf{logarithmic manifold}, see \cite[Thm. A]{KU08}.
\item[(ii). ] Moreover, when we assume the completeness of the fan $\Sigma$ (to be defined in the next section), $X_{\Sigma}=\Gamma\backslash D_{\Sigma}$ --- while itself only a ``partial compactification'' in general due to the aforementioned slits --- provides a compactification for any (closed) variation of Hodge structure inside $\Gamma\backslash D$. 
\end{description}
\end{remark}

Kato and Usui proved the fans constructed in \cite{AMRT10}, which produce (full) toroidal compactifications in the classical cases, are complete in their sense.  They also conjectured the existence of a complete fan for general cases. However, \cite{Wat08} gives a counterexample which disproved the conjecture. In the next section, we will disprove the existence of a complete fan in somewhat more general settings.

\section{Kato-Usui's Complete Fan Conjecture}\label{S4}

In order to generalize the signature features of toroidal compactification introduced in \cite{AMRT10} to the non-classical setting, Kato-Usui introduced the concept of complete fan for non-classical period domains and conjectured the existence of such a fan, see \cite{KU08}. They also proved the conjecture is true for the classical case introduced in \cite{AMRT10}. Intuitively speaking, a complete fan consists of all stories about nilpotent orbits, and therefore compactifies every period map of a given type.

The conjecture is already proved to be false:  in \cite{Wat08}, Watanabe gave a counterexample with period domain of Hodge type $(2,2,2)$. In this section, we are going to investigate the conjecture. Moreover, we give a proof that the conjecture fails for a wider classes of period domains.  Based on the proof, we will give two examples in the \textbf{Appendix} which illustrate how the proof works for certain types of non-classical period domains. One of these two examples will come from Watanabe's paper \cite{Wat08}.

Continuing with the notations of previous sections, we need a sequence of definitions and properties from \cite{KU08}.

\begin{definition}

\begin{description}
\item[(1) ]$\mathcal{V}:=$ $\{(A, V) |$  A is a $\mathbb{Q}$-linear subspace of $\mathfrak{g}_\mathbb{Q}$ consisting of mutually commutative nilpotent elements, $V$ is a submonoid of $A^*:=Hom_\mathbb{Q}(A, \mathbb{Q})$ such that $V\cap (-V)=\{0\}$ and $V\cup (-V)=A^*$\}.
\item[(2)]Given $(A,V)\in \mathcal{V}$, let $\mathcal{F}(A,V)$ be the set of rational nilpotent cones $\sigma \subset \mathfrak{g}_{\mathbb{R}}$ such that $\sigma_\mathbb{R}=A_\mathbb{R}$ and $(\sigma \cap A)^{\vee}:= \{h \in A^* | h(\sigma \cap A) \subset \mathbb{Q}_{\geq 0}\} \subset V$.
\end{description}

\end{definition}

\begin{definition}

\begin{description}
\item[(1) ] $\check{D}_{val}:=\{(A,V,Z)|(A,V)\in \mathcal{V}$ and $Z$ is an $exp(A_{\mathbb{C}})$-orbit in $\check{D}$\}
\item[(2) ] ${D}_{val}:=\{(A,V,Z)|(A,V,Z)\in \check{D}_{val}$  and there exists $\sigma \in \mathcal{F}(A,V)$ such that $Z$ is a $\sigma$-nilpotent orbit\}
\end{description}

\end{definition}

\begin{definition}

\begin{description}
\item[(1) ] Suppose $\Sigma$ is a fan in $\mathfrak{g}_{\mathbb{Q}}$. For $(A,V) \in \mathcal{V}$, let 
\begin{center}
    $X_{A,V,\Sigma}:=\{\sigma \subset \Sigma | \sigma \cap A_{\mathbb{R}} \in \mathcal{F}(A,V)\}$
\end{center}
By \cite{KU08}, whenever $X_{A,V,\Sigma}$ is not empty, there is a minimal element $\sigma_0 \in X_{A,V,\Sigma}$.  
\item[(2) ] Suppose $\Sigma$ is a fan in $\mathfrak{g}_{\mathbb{Q}}$, define:
\begin{center}
    $D_{\Sigma, val}:=\{(A,V,Z) \in \check{D}_{val} | X_{A,V,\Sigma}$ is non-empty, and $exp(\sigma_{0,\mathbb{C}})Z$ is a $\sigma_0$-nilpotent orbit\}
\end{center}
\item[(3) ] We say a fan $\Sigma$ in $\mathfrak{g}_{\mathbb{Q}}$ is \textbf{complete} with respect to $D$ if $D_{val} = D_{\Sigma, val}$.
\end{description}

\end{definition}

A complete fan $\Sigma$ for the period domain of given type, if it exists, should contain universal information about nilpotent cones coming from nilpotent orbits. When the maximal nilpotent orbits are all $1$-dimensional, the construction and verification of a complete fan is clear: we just take the union of all rational nilpotent elements which polarize some nilpotent orbits.

When there are nilpotent orbits of dimension greater than one, this means that two nilpotent cones could intersect at some interior points of both.  While one can subdivide a cone to eliminate finitely many such interior intersections, the problem is that completeness may force infinitely many interior intersections, so that ``no subdivision is enough''.  This phenomenon could present a fatal strike against the existence of a complete fan for these cases.  In the rest of this section, we will prove that this is what happens in many non-classical situations.

\bigskip
Suppose $e^{\sigma_{\mathbb{C}}} F^{\bullet}$ is a nilpotent orbit of type $\{h^{p,q}\}_{p+q=n}$ and dimension $k\geq 2$.  Here
\begin{equation}
    \sigma = \mathbb{R}_{\geq 0}<N_1, \ldots , N_k>\;\;\;\text{and}\;\;\;\sigma_{\mathbb{C}}=\mathbb{C}<N_1,\ldots, N_k>,
\end{equation}
where $\{N_i\}_{1\leq i \leq k}$ are commuting nilpotent elements.  For $\vec{v}:= (v_1, ... , v_k)\in \mathbb{Q}_{\geq 0}^k$, let $N_{\vec{v}}:=\sum_{i=1}^{k}v_iN_i$ and $\sigma_{\vec{v}}:=\mathbb{R}_{\geq 0}<N_{\vec{v}}>.$

After a rescaling, we may assume that $F^{\bullet}\in D$.  Notice that $F^{\bullet}$ induces a weight-$0$ Hodge structure on $\mathfrak{g}:= End(H,Q)$, which is represented by the Hodge filtration $\tilde{F}^{\bullet}$ on $\mathfrak{g}$ as:

\begin{equation}
    \{0\}\subset \tilde{F}^{n}\subset \tilde{F}^{n-1}\subset ... \subset \tilde{F}^{-n+1}\subset \tilde{F}^{-n} = \mathfrak{g_{\mathbb{C}}}
\end{equation}
where $\tilde{F}^{p}:=\{f\in End(H_{\mathbb{C}}, Q)|f(F^k)\subset F^{k+p})\}$. It is clear that all $N_i$ and $N_{\vec{v}}$ are in $\tilde{F}^{-1}$.  Let

\begin{equation}
C^h(N_{\vec{v}}):=\{M\in \mathfrak{g_{\mathbb{C}}} | [M,N_{\vec{v}}] \in \tilde{F}^{-1}\},
\end{equation}
\begin{equation}
C^h(\sigma) := \cap_{i=1}^k C^h(N_i).
\end{equation}
Then we have
\begin{equation}
    C^h(N_{\vec{v}}) = \ker(\mathrm{ad}(N_{\vec{v}}))+\tilde{F}^0
\end{equation}
and
\begin{equation}\label{eq4.6}
    C^h(\sigma) \subseteq  C^h(N_{\vec{v}}).
\end{equation}
For classical cases, the two sides in \textbf{\eqref{eq4.6}} are always equal. The next lemma concerns the case when \textbf{\eqref{eq4.6}} is a strict inclusion.

\begin{lemma}\label{lem4.4}
Suppose $exp(\sigma_{\mathbb{C}})F^{\bullet}$ is a monodromy nilpotent orbit, and for some $\vec{v} \in \mathbb{Q}_{> 0}^k$ (equiv. $N_{\vec{v}}\in int(\sigma)$), we have $C^h(\sigma) \subsetneqq C^h(N_{\vec{v}})$.  Then there exists $F_{\vec{v}}^{\bullet} \in \check{D}$ such that $exp(\mathbb{C}N_{\vec{v}}) F_{\vec{v}}^{\bullet}$ is a $\sigma_{\vec{v}}$-nilpotent orbit, but $exp(\sigma_{\mathbb{C}}) F_{\vec{v}}^{\bullet}$ is not horizontal.
\end{lemma}

\begin{proof}
We may assume $M\in C^h(N_{\vec{v}})\setminus  C^h(\sigma)$ is real, commutes with $N_{\vec{v}}$, and is close to $\{0\}$.  So writing $g=e^M$, we have $e^{zN_{\vec{v}}}g^{-1}F^{\bullet}=g^{-1}e^{z N_{\vec{v}}}F^{\bullet}\in D$ for $\mathfrak{Im}(z)\gg 0$.  Setting $F_{\vec{v}}^{\bullet}:=g^{-1}F^{\bullet}$, $e^{\mathbb{C} N_{\vec{v}}}F^{\bullet}_{\vec{v}}$ is clearly a nilpotent orbit.

Since $M\notin C^h(\sigma)$, there exists $N'\in int(\sigma)$ such that $\mathrm{ad}(M)N'\notin \tilde{F}^{-1}$.  So $e^{z\mathrm{ad}(M)N'}F^{\bullet}=ge^{zN'}F_{\vec{v}}^{\bullet}$ hence $e^{zN'}F_{\vec{v}}^{\bullet}\,(\subset e^{\sigma_{\mathbb{C}}}F_{\vec{v}}^{\bullet})$ is not horizontal.
%Firstly, it is clear to see for $g \in G = Sp(6, \mathbb{C})$ sitting inside some neighborhood of the identity, $exp(\mathbb{C}N_{\vec{v}}) gF^{\bullet}$ is a nilpotent orbit if and only if $exp(\mathbb{C}(gN_{\vec{v}}g^{-1})) F^{\bullet}$ is a nilpotent orbit. The Griffiths transversality condition implies $gNg^{-1} \in \tilde{F}^{-1}$. For those $g\in G = Sp(6, \mathbb{C})$ which can be written as $g = exp(N)$ for some $N \in \tilde{F}^{-1}$, it is equivalent to say $N \in C^h(N_{\vec{v}})$.
%As the positivity condition in \textbf{Definition 3.6.(iii)} is an open condition, by the assumption, we can choose such a $\vec{v} \in int(\sigma)$ and $N \in C^h(N_{\vec{v}}) \backslash C^h(\sigma)$ such that $exp(\mathbb{C}N)F^{\bullet}$ is a $exp(\mathbb{C}N)$-nilpotent orbit. Let $g = exp(N)$ and $F_{\vec{v}}^{\bullet}=gF^{\bullet} \in \check{D}$. It is clear the choice $F_{\vec{v}}^{\bullet}$ satisfies the requirement.
\end{proof}

\begin{remark}\label{rem4.5}
Given one $\vec{v}\in \mathbb{Q}_{>0}^k$ satisfying the hypothesis of the Lemma, ``almost all'' $\vec{v}'\in \mathbb{Q}_{>0}^k$ satisfy it.  More precisely, the $\{N_{\vec{v}'}\}$ satisfying the hypothesis are analytically dense in $\sigma$.  

To see this, note that $int(\sigma)$ is contained in the orbit of a subgroup $L\leq G_{\mathbb{R}}$ stabilizing $F^{\bullet}$ (and $\tilde{F}^{\bullet}$) \cite[Thm. 3.1]{KPR19}; so we can pick $g\in L$ with $\mathrm{Ad}(g)N_{\vec{v}}=N_{\vec{v}'}$.  With $M$ as in the Lemma's proof, set $M':=\mathrm{Ad}(g)M$.  Then we have $[M',N_{\vec{v}'}]=[\mathrm{Ad}(g)M,\mathrm{Ad}(g)N_{\vec{v}}]=\mathrm{Ad}(g)[M,N_{\vec{v}}]\in \mathrm{Ad}(g)\tilde{F}^{-1}=\tilde{F}^{-1}$ hence $M'\in C^h(N_{\vec{v}'})$.  On the other hand, we had $\mathrm{ad}(M)\sigma \not\subset \tilde{F}^{-1}$ in the proof, and $\mathrm{ad}(M')\sigma\not\subset \tilde{F}^{-1}$ is thus a nonempty Zariski open condition on $M'$, $g$, and $N_{\vec{v}}$.
% When $N$ is defined over $\mathbb{R}$, the resulting $F_{\vec{v}}^{\bullet}$ is also an element in $D$.
\end{remark}

\begin{theorem}\label{thm4.6}
Suppose there is a $k \geq 2$-dimensional nilpotent orbit $e^{\sigma_{\mathbb{C}}}F^{\bullet}\subset \check{D}$ satisfying the assumption in \textbf{Lemma \ref{lem4.4}} (that is, that there exists a rational $N_{\vec{v}}\in int(\sigma)$ with $C^h(\sigma)\subsetneq C^h(N_{\vec{v}})$). Then no fan $\Sigma$ in $\mathfrak{g}_{\mathbb{Q}}$ is complete with respect to $D$.
\end{theorem}

\begin{proof}
Write $\sigma_{\mathbb{Q}}$ for the $\mathbb{Q}$-span of $\sigma$'s generators.  Put $A:=\sigma_{\mathbb{Q}}$, $Z:=e^{\sigma_{\mathbb{C}}}F^{\bullet}$, and choose $V\subset A^*$ containing $(\sigma\cap A)^{\vee}$.  Then $(A,V,Z)\in D_{val}$.

Suppose $\Sigma$ is complete; then $(A,V,Z)$ must belong to $D_{\Sigma,val}$.  So there exists a (minimal) $\tau\in \Sigma$ with $\tau\cap A_{\mathbb{R}}\in \mathcal{F}(A,V)$ ($\implies \tau_{\mathbb{C}}\supseteq \sigma_{\mathbb{C}}$), and $e^{\tau_{\mathbb{C}}}Z\,(=e^{\tau_{\mathbb{C}}}F^{\bullet})$ is a $\tau$-nilpotent orbit.  By \cite[Lem. 5.3.4]{KU08}, $Z$ is a $(\tau\cap A_{\mathbb{R}})$-nilpotent orbit and $\tau\cap A_{\mathbb{R}}$ contains interior points of $\tau$.  Moreover, since $(\tau\cap A)^{\vee}$ and $(\sigma\cap A)^{\vee}$ both belong to $V$, $\tau\cap A_{\mathbb{R}}$ and $\sigma$, hence $\tau$ and $\sigma$, have common interior points.

By the assumption and \textbf{Remark \ref{rem4.5}}, there must exist $N_{\vec{v}}\in (int(\sigma)\cap \sigma_{\mathbb{Q}})\cap int(\tau)$ which satisfies the hypothesis of \textbf{Lemma \ref{lem4.4}}.  By this Lemma, there must also exist $F^{\bullet}_{\vec{v}}\in \check{D}$ such that $Z_{\vec{v}}:=e^{\mathbb{C}N_{\vec{v}}}F_{\vec{v}}^{\bullet}$ is a $\sigma_{\vec{v}}$-nilpotent orbit, but $e^{\sigma_{\mathbb{C}}}F_{\vec{v}}^{\bullet}$ is not horizontal.  Writing $A_{\vec{v}}:=\mathbb{Q}<N_{\vec{v}}>\,(\cong \mathbb{Q})$ and $(\mathbb{Q}_{\geq 0}\cong)\,(\sigma_{\vec{v}}\cap A_{\vec{v}})^{\vee}=:V_{\vec{v}}\subset A_{\vec{v}}^*$, we have $(A_{\vec{v}},V_{\vec{v}},Z_{\vec{v}})\in D_{val}=D_{\Sigma,val}$.  So there exists a (minimal) $\tau_{\vec{v}}\in \Sigma$ such that $\tau_{\vec{v}}\cap A_{\vec{v},\mathbb{R}}=\sigma_{\vec{v}}$, and $e^{\tau_{\vec{v}},\mathbb{C}}Z_{\vec{v}}$ is a $\tau_{\vec{v}}$-nilpotent orbit.

Finally, since $\tau_{\vec{v}}$ contains an interior point of $\tau$, and both are in $\Sigma$, and $\Sigma$ is a fan, we must have $\tau_{\vec{v}}\supseteq \tau$.  Hence $e^{\tau_{\mathbb{C}}}Z_{\vec{v}}$ is horizontal.  But then $e^{\sigma_{\mathbb{C}}}F^{\bullet}_{\vec{v}}$ is horizontal, which is a contradiction.
\end{proof}

We now apply this result to some particular period domains.  Recall that the period domain of Hodge structures with Hodge numbers increasing toward the center, $h^{n,0}\leq h^{n-1,1}\leq \cdots\leq h^{n-\lfloor\frac{n}{2}\rfloor,\lfloor\frac{n}{2}\rfloor}$, always admit Hodge-Tate degenerations, see \cite[Chap. V]{GGR14}.  When this is the case, the maximal dimension of a nilpotent orbit in $D$ is the same as the maximal dimension of a variation of Hodge structure.\footnote{We may assume the limit mixed Hodge structure is $\mathbb{R}$-split, so that the group $G^{0,0}$ with Lie algebra $\mathfrak{g}^{0,0}$ in the Deligne bigrading is defined over $\mathbb{R}$.  Denoting the monodromy logarithm by $N$, we have $\mathrm{ad}(N)\colon \mathfrak{g}^{0,0}_{\mathbb{R}}\twoheadrightarrow \mathfrak{g}^{-1,-1}_{\mathbb{R}}$.  It follows that $(\mathrm{ad}\mathfrak{g}^{0,0}_{\mathbb{R}})N=\mathfrak{g}^{-1,-1}_{\mathbb{R}}$, and so $\mathrm{Ad}(G^{0,0}_{\mathbb{R}})N$ is an open cone in $\mathfrak{g}_{\mathbb{R}}^{-1,-1}$.  Now take any rational simplicial $\sigma$ contained in this cone, containing $N$, and abelian in $\mathfrak{g}_{\mathbb{R}}$.  The maximal (real) dimension of such a $\sigma$ (hence complex dimension of $\sigma_{\mathbb{C}}$) is the same as the maximal (complex) dimension of an abelian subalgebra of $\mathfrak{g}^{-1,1}$ in the Hodge decomposition of $\mathfrak{g}_{\mathbb{C}}$ at a point on the nilpotent orbit in $D$.}  We will actually show that for \emph{any} $N_{\vec{v}}$ in the interior of a corresponding maximal-dimensional $\sigma$, the hypothesis of \textbf{Lemma \ref{lem4.4}} holds.  (So we could have avoided \textbf{Remark \ref{rem4.5}} and made do with a weaker statement in \textbf{Theorem \ref{thm4.6}}.  The stronger statement is really only used in the \textbf{Appendix}, since it is convenient for explicit examples.)

For general $N \in \tilde{F}^{-1}$, 
\begin{equation}
    \mathrm{ad}(N): \tilde{F}^{-1}\rightarrow \tilde{F}^{-2}
\end{equation}
induces a natural map
\begin{equation}
    \varphi_N: \tilde{F}^{-1}/ \tilde{F}^{0}\rightarrow \tilde{F}^{-2}/\tilde{F}^{-1}.
\end{equation}
Therefore, $$\dim(ker(\varphi_N))\geq 2\dim(\tilde{F}^{-1})-\dim(\tilde{F}^{-2})-\dim(\tilde{F}^{0})=\dim\mathfrak{g}_{-1}-\dim\mathfrak{g}_{-2}$$ where $\mathfrak{g}_i:=\tilde{F}^{i}/\tilde{F}^{i+1}$. For period domains of Hodge type $h_{S} = (a, b, a), (b\geq a\geq 2)$ and $h_{CY3} = (1,m,m,1)$, which correspond to the Hodge type of surfaces of general type and Calabi-Yau 3-folds respectively, Carlson showed in \cite{Car86} that the maximal dimension of variation of Hodge structures, which are equivalent to maximal dimension of monodromy nilpotent orbits, are $\lfloor\frac{1}{2}ab\rfloor$ and $m$, respectively. 

So let $\sigma$ be a nilpotent cone of maximal dimension (of Hodge-Tate type), and $N\in int(\sigma)$.  For Hodge type $h_S$ and $h_{CY3}$, (dim$\mathfrak{g}_{-1}$, dim$\mathfrak{g}_{-2}$) are $(ab,\tfrac{a(a-1)}{2})$ and $(\tfrac{m(m+3)}{2}, m)$ respectively. We have $ab-\tfrac{a(a-1)}{2}> \lfloor\frac{1}{2}ab\rfloor$ for $b\geq a$ and $\tfrac{m(m+3)}{2}-m> m$ for $m\geq 2$.  So in these cases, we conclude that 
\begin{equation}\label{eq4.9}
\dim(ker(\varphi_N))>\dim(\sigma_{\mathbb{C}}).
\end{equation}
But by maximal-dimensionality of $\sigma$, we have $C^h(\sigma)\cap \tilde{F}^{-1}=\sigma_{\mathbb{C}}+\tilde{F}^0$.  Thus, \eqref{eq4.9} shows that $\sigma_{\mathbb{C}}+\tilde{F}^0\subsetneq ker(\varphi_{N})+\tilde{F}^0 (\subset C^h(N)\cap \tilde{F}^{-1})$, so that $C^h(\sigma)\cap \tilde{F}^{-1}\subsetneq C^h(N)\cap \tilde{F}^{-1}$ and the hypothesis of \textbf{Lemma \ref{lem4.4}} holds.  Thus we have proved the following theorem:

\begin{theorem}\label{thm4.7}
%Let $D$ be the period domain parametrizes Hodge structures on $H_{\mathbb{Z}}$ of certain type.
%\begin{description}
%\item[(1). ]Suppose there exists a (rational) nilpotent orbit $exp(\sigma_{\mathbb{C}})F^{\bullet}$ where $\sigma = \mathbb{Q}_{\geq 0}<N_1, N_2, ... ,N_k>$ with $k\geq 2$ and $N_i$ are commuting nilpotent elements in $End(H_{\mathbb{Q}})$, such that there exists $1\leq i \neq j \leq k$ and $C(N_i)\neq C(N_j)$, then there exists no complete fan.
%\item[(2). ]Particularly,
For period domains $D$ of type $h=(a,b,a), b\geq a\geq 2$ and type $h = (1, m, m, 1), m\geq 2$, there is no complete fan.
%\end{description}
\end{theorem}

\noindent In the \textbf{Appendix}, we will provide two explicit examples illustrating this failure.

\medskip

After Watanabe provided the example in \cite{Wat08} which disproved the complete fan conjecture, Kato and Usui modified the conjecture and expected the modified versions to be true. Generally speaking, the modified conjecture requires the definition of a fan to be replaced by a weak fan which is defined in \textbf{Remark \ref{rem3.2}}, as well as slight modifications on the definition of completeness, see \cite[Sec. 7.3]{KNU10} for more details. This modified conjecture remains open.

\medskip

One can not provide a universal compactification for images of period maps in a nonclassical $\Gamma\backslash D$ by simply taking the union of all nilpotent cones and their $\Gamma$-orbits, because of the failure of complete fan conjecture. However, for a certain variation of Hodge structure coming from some geometric family, it is still possible to construct a $\Gamma$-strongly compatible fan based on given (finitely generated) monodromy information. In the following section, we will take a close look on a specific family coming from a family of Complete Intersection Calabi-Yau (CICY) threefolds studied by Hosono and Takagi. We will make use of computational results in \cite{HT14} and \cite{HT18} and show how to construct a $\Gamma$-strongly compatible fan $\Sigma$ for some properly chosen $\Gamma$ in this case, which, as far as the author knows, is the first such construction for a non-classical family of dimension $>1$.

\section{A toric complete intersection family}\label{S5}

In \cite{HT14} and \cite{HT18}, Hosono and Takagi studied a family of generic Complete Intersection Calabi-Yau (CICY) manifolds in a product of projective spaces. Regarding the general constructions of such families as well as their geometric and arithmetic properties, one may consult \cite{CDLS88} or \cite{AAGGL16}. In this paper, we will investigate this family further.

\medskip

The family is given by complete intersection Calabi-Yau manifolds of the form:

\begin{equation}
X:= \begin{pmatrix}
\mathbb{P}^4 & | & 1 & 1 & 1 & 1 & 1\\
\mathbb{P}^4 & | & 1 & 1 & 1 & 1 & 1
\end{pmatrix} ^{(2,52)} 
\end{equation} 
which is an intersection of $5$ general $(1,1)$-divisors in $\mathbb{P}^4 \times \mathbb{P}^4$ and has Hodge numbers $(h^{1,1},h^{2,1})=(2,52)$. The $A$-structure on the even cohomology of this family is investigated in \cite{HT14}. In \cite{HT18}, the mirror family $\mathcal{X}^*$, whose middle cohomology which describes the corresponding $B$-structure, is constructed as follows. 

Define the following special family
\begin{equation}
X_{sp}:= \{z_iw_i+az_iw_{i+1}+bz_{i+1}w_i=0, i= 1,...,5\},
\end{equation}
where $z_i, w_j$ are coordinates of two $\mathbb{P}^4$ respectively, and the indexes $i, j$ are considered as elements in $\mathbb{Z}_5$.

\begin{prop}\label{prop5.1}
\textbf{\cite[Prop. 3.11]{HT18}} For general values of $a, b$, we have the following properties:
\begin{enumerate}
    \item $X_{sp}$ is singular along 20 lines with singularity of $A_1$ type.
    \item There exists a crepant resolution $X^{*} \rightarrow X_{sp}$ with $X^{*}$ being a Calabi-Yau threefold with $h^{1,1}(X^{*})=52, h^{2,1}(X^{*})=2$.
    \item The resolution parametrized by $(a,b) \in \mathbb{C}^2$ gives a family over the space $U=\bar{\mathcal{M}}_{X^*}^{cpx} \backslash Dis $ where $\bar{\mathcal{M}}_{X^*}^{cpx} = \mathbb{P}^2$, and $Dis = D_1 \bigcup D_2 \bigcup D_3 \bigcup Dis_0$ where $D_i$ are coordinate hyperplanes and $Dis_0$ is an irreducible nodal curve of degree 5. 
\end{enumerate}
\end{prop}

The nodal curve $Dis_0$ contains $6$ self-intersection nodes, and has one fifth-order tangency with each coordinate hyperplane, labeled $t_1,t_2,t_3$. Denote the intersection points of coordinate hyperplanes by $o_1, o_2, o_3$.  (We refer to \cite[Fig. 6.1]{HT14} for a picture of the discriminant.) 

Around each singular point of the discriminant locus, there is a canonical integral symplectic structure arising from the solution of local Picard-Fuchs equations. More precisely, denote the family as $\pi: \mathcal{X}^* \rightarrow B$, and fix a meromorphic section of $R^0\pi_*\Omega^3_{\mathcal{X}^*}$, namely $\Omega$. Locally, the map $b\in B \mapsto \vec{V_b}:= (\int_{\gamma_i(b)}\Omega_b)$ gives a basis of local solutions to the Picard-Fuchs equation for some (integral, symplectic) basis ${\gamma_i(b)}$ of $H_3(X_b,\mathbb{Q})$. We also have the action by local monodromy operators on the period matrix $\vec{V_b}$. There are many references in the literature, for example, \cite{HLY95}.

Let $b_0$ be a smooth point near the boundary point $o_1$, and denote by $T_x,T_y$ the monodromy operators on $H^3(X_{b_0}^*,\mathbb{Q})$ given by small loops (based at $b_0$) about $x=0$ and $y=0$.  According to \cite{HT18}, there is an ordered integral (symplectic) basis $\{\alpha_0, \alpha_1, \alpha_2, \beta_2, \beta_1, \beta_0\} \subset H^3(X^*_{b_0},\mathbb{Q})$ in terms of which the (dual) intersection matrix is given by
\small
\begin{equation}\label{eq5.3}
    Q:=
\begin{pmatrix}
0 & 0 & 0 & 0 & 0 & -1\\
0 & 0 & 0 & 0 & -1 & 0\\
0 & 0 & 0 & -1 & 0 & 0\\
0 & 0 & 1 & 0 & 0 & 0\\
0 & 1 & 0 & 0 & 0 & 0\\
1 & 0 & 0 & 0 & 0 & 0
\end{pmatrix},
\end{equation}
\normalsize
and the monodromy operators become
\small
\begin{equation}
    T_x=
\begin{pmatrix}
1 & -1 & 0 & 5 & 3 & 5\\
0 & 1 & 0 & -10 & -5 & -2\\
0 & 0 & 1 & -10 & -10 & -5\\
0 & 0 & 0 & 1 & 0 & 0\\
0 & 0 & 0 & 0 & 1 & 1\\
0 & 0 & 0 & 0 & 0 & 1
\end{pmatrix},\;\;\;
T_y=
\begin{pmatrix}
1 & 0 & -1 & 3 & 5 & 5\\
0 & 1 & 0 & -10 & -10 & -5\\
0 & 0 & 1 & -5 & -10 & -2\\
0 & 0 & 0 & 1 & 0 & 1\\
0 & 0 & 0 & 0 & 1 & 0\\
0 & 0 & 0 & 0 & 0 & 1
\end{pmatrix}.
\end{equation}
\normalsize
The discriminant locus $Dis_0$ touches each coordinate line at a fifth-order tangent point, which provides an $A_9$ singularity. Resolving any one of these singularities by blowing up $5$ times, we obtain a sequence of exceptional divisors $E_i$, $i=1,...,5$. By \cite[Prop. 4.6]{HT18}, the monodromy matrix corresponding to $E_1$ is given by:
\small
\begin{equation}
    T_{E_1}=
\begin{pmatrix}
1 & 0 & 0 & 0 & 0 & 0\\
0 & 1 & 0 & 0 & 50 & 0\\
0 & 0 & 1 & 0 & 0 & 0\\
0 & 0 & 0 & 1 & 0 & 0\\
0 & 0 & 0 & 0 & 1 & 0\\
0 & 0 & 0 & 0 & 0 & 1
\end{pmatrix}
\end{equation}
\normalsize
\begin{lemma}\label{lem5.2}
For $i=1,...,5$, up to constants, all $T_{E_i}$ have the same unilpotent monodromy logarithm.
\end{lemma}

\begin{proof}
See \textbf{Section \ref{S6}}.
\end{proof}

By the above lemma, it is enough to consider $T_{E_1}$ for the purpose of computing monodromy logarithms of all $T_{E_i}$. We use the letter ``$N$'' to replace ``$T$'' with the same sub-index to denote the corresponding unipotent monodromy logarithm (that is, the logarithm of the unipotent part). For instance, $T_x=T_x^{ss}e^{N_x}$ where $T_x^{ss}$ represents the semisimple part.

\medskip

The blown-up discriminant locus has normal-crossing points at $x=y=0$ and at a point on the $y$-axis over the 5th-order tangency.  Denote the resulting monodromy cones by
\begin{equation}
\sigma_1:=<N_x, N_y>\;\;\text{and}\;\; \sigma_2:=<N_y, N_{E_1}>.
\end{equation}
In addition, there are $6$ self-intersection nodes on the conifold curve $Dis_0$, which are all ordinary double points. Choosing $2$ of these nodes, together with paths from points near the nodes to $b_0$, we may write the corresponding local monodromy nilpotent cones with respect to our basis at $b_0$, viz.
\begin{equation}
\sigma_3:=<N_{p_1,1}, N_{p_1,2}>, \;\;\;\sigma_4:=<N_{p_3,1}, N_{p_3,2}>.
\end{equation}
By \cite[Prop. 6.6-6.8]{HT14}, the monodromies are given explicitly (all in the same basis) by

\tiny
\begin{equation}\label{eq5.8}
N_x=
\begin{pmatrix}
0 & -1 & 0 & 0 & \frac{1}{2} & \frac{25}{6}\\
0 & 0 & 0 & -10 & -5 & \frac{1}{2}\\
0 & 0 & 0 & -10 & -10 & 0\\
0 & 0 & 0 & 0 & 0 & 0\\
0 & 0 & 0 & 0 & 0 & 1\\
0 & 0 & 0 & 0 & 0 & 0
\end{pmatrix},
N_y=
\begin{pmatrix}
0 & 0 & -1 & \frac{1}{2} & 0 & \frac{25}{6}\\
0 & 0 & 0 & -10 & -10 & 0\\
0 & 0 & 0 & -5 & -10 & \frac{1}{2}\\
0 & 0 & 0 & 0 & 0 & 1\\
0 & 0 & 0 & 0 & 0 & 0\\
0 & 0 & 0 & 0 & 0 & 0
\end{pmatrix},
\end{equation}

\begin{equation}\label{eq5.9}
N_{E_1}=
\begin{pmatrix}
0 & 0 & 0 & 0 & 0 & 0\\
0 & 0 & 0 & 0 & 50 & 0\\
0 & 0 & 0 & 0 & 0 & 0\\
0 & 0 & 0 & 0 & 0 & 0\\
0 & 0 & 0 & 0 & 0 & 0\\
0 & 0 & 0 & 0 & 0 & 0
\end{pmatrix}.
\end{equation}

\begin{equation}\label{eq5.10}
N_{p_1,1}=
\begin{pmatrix}
5 & 0 & 5 & 10 & 25 & -25\\
-5 & 0 & -5 & -10 & -25 & 25\\
-2 & 0 & -2 & -4 & -10 & 10\\
1 & 0 & 1 & 2 & 5 & -5\\
0 & 0 & 0 & 0 & 0 & 0\\
1 & 0 & 1 & 2 & 5 & -5
\end{pmatrix},
N_{p_1,2}=
\begin{pmatrix}
5 & 5 & 0 & 25 & 10 & -25\\
-2 & -2 & 0 & -10 & -4 & 10\\
-5 & -5 & 0 & -25 & -10 & 25\\
0 & 0 & 0 & 0 & 0 & 0\\
1 & 1 & 0 & 5 & 2 & -5\\
1 & 1 & 0 & 5 & 2 & -5
\end{pmatrix}.
\end{equation}

\begin{equation}\label{eq5.11}
N_{p_3,1}=
\begin{pmatrix}
40 & 30 & 0 & 50 & 10 & -100\\
-4 & -3 & 0 & -5 & -1 & 10\\
-20 & -15 & 0 & -25 & -5 & 50\\
0 & 0 & 0 & 0 & 0 & 0\\
12 & 9 & 0 & 15 & 3 & -30\\
16 & 12 & 0 & 20 & 4 & -40
\end{pmatrix},
N_{p_3,2}=
\begin{pmatrix}
40 & 0 & 30 & 10 & 50 & -100\\
-20 & 0 & -15 & -5 & -25 & 50\\
-4 & 0 & -3 & -1 & -5 & 10\\
12 & 0 & 9 & 3 & 15 & -30\\
0 & 0 & 0 & 0 & 0 & 0\\
16 & 0 & 12 & 4 & 20 & -40
\end{pmatrix}.
\end{equation}
\normalsize
Moreover, the cones $\sigma_i, i=1,2,3,4$ and their $Ad(Sp(6,\mathbb{Z}))$-orbits give all information about the monodromy of this family, see \cite[Prop 6.8(d)]{HT14}.  (In particular, the other $8$ cones are in this orbit.)  We denote the monodromy group of this family by $\Gamma_0$. Suppose $\Gamma\leq Sp(6, \mathbb{Z})$ is a subgroup of finite index. We consider all rational nilpotent cones coming from the boundary points of this family, as well as their $Ad(\Gamma)$-orbits in $\mathfrak{g}_{\mathbb{Q}}$, and denote this collection of nilpotent cones by $\Sigma$. We consider the space $X_{\Sigma}=\Gamma\backslash D_{\Sigma}$ defined in Section 4. To fit the definition, we need to prove $\Sigma$ is indeed a fan; i.e., it should satisfy conditions \textbf{(i), (ii)} in \textbf{Definition \ref{def3.1}}. Notice that if $\Sigma$ is proved to be a fan, its $\Gamma$-strong compatibility will be automatically guaranteed.

\begin{theorem}\label{thm5.3}
There exists a finite-index subgroup $\Gamma\leq Sp(6,\mathbb{Z})$, such that the cones $\sigma_i, i=1,2,3,4$ and their $Ad(\Gamma)$-orbits glue up to a well-defined fan $\Sigma$ in $\mathfrak{g}_{\mathbb{Q}}=\mathfrak{sp}(6,\mathbb{Q})$, and it is also strongly compatible with $\Gamma$.
\end{theorem}
\begin{proof}
See \textbf{Section \ref{S6}}.  The main point is to rule out infinitely many interior intersections between a given cone $\sigma$ and the $\mathrm{Ad}(\Gamma)$-orbits of both $\sigma$ and the other cones.
\end{proof}

Notice $\Gamma\bigcap \Gamma_0$ is a finite-index subgroup of $\Gamma_0$. Up to passing to a finite cover we can assume $\Gamma_0\leq \Gamma$. It is clear from \textbf{Theorem \ref{thm5.3}} that the subcollection $\Sigma_0\subset \Sigma$ obtained by taking $Ad(\Gamma_0)$-conjugates of monodromy nilpotent cones is a $\Gamma_0$-strongly compatible fan.  Denote by $\widetilde{\mathbb{P}^2}$ the result of the blow-ups at the $\{p_i\}$ together with the finite cover just mentioned, and $\tilde{U}\subset\widetilde{\mathbb{P}^2}$ the complement of the discriminant's preimage.  Then by Kato-Usui's construction we have an extended period map 
\begin{equation}\label{eq5.12}
\widetilde{\mathbb{P}^2}\rightarrow \Gamma_0\backslash D_{\Sigma_0}
\end{equation}
into a partial compactification, restricting on $\tilde{U}$ to the period map for (a base-change of) the family in \textbf{Proposition \ref{prop5.1}}.

To prove this compactification construction is non-trivial, in the sense that $\widetilde{\mathbb{P}^2} \rightarrow \Gamma_0\backslash D_{\Sigma_0}$ does not factor through (a partial compactification of) any proper Mumford-Tate subdomain, we need to show the geometric variation is "generic" enough.  (For example, if it did factor through a subdomain of classical type, the compactification in \eqref{eq5.12} would be essentially the toroidal compactification of Mumford et al; while if it factored through a product domain for direct sums of Hodge structures of type $(1,1,1,1)$ and $(0,1,1,0)$, one could take a product of partial compactifications.)  We will only state the key fact here.  More details including the proof will be presented in \textbf{Section \ref{S7}}.

\begin{theorem}\label{thm5.4}
The geometric variation associated to the one in \textbf{Proposition \ref{prop5.1}} is Mumford-Tate generic, i.e. its image in $D$ contains (a class represented by) a Mumford-Tate generic point (Definition \ref{def7.1}), meaning the compactification given in \eqref{eq5.12} can not factor through a partial compactification of a proper Mumford-Tate subdomain of $D$.
\end{theorem}

\section{Proof of Lemma 5.2 and Theorem 5.3}\label{S6}

\subsection{Proof of Lemma 5.2}

Consider the smooth toric variety $W$ given by the fan in $\mathbb{R}^2$ depicted in Figure 6.1,
\begin{equation*}
\begin{tikzpicture}[scale=1]
\node at (0,0) {};
\draw [-stealth](0,0) -- (1,0) [line width = 1pt];
\node at (2,0) {$v_3=(1,0)$};
\draw [-stealth](0,0) -- (1.5,1) [line width = 1pt];
\node at (1.7,1.2) {$v_1=(a,1)$};
\draw [-stealth](0,0) -- (2,-1) [line width = 1pt];
\node at (2.2,-1.2) {$v_2=(b,-1)$};
\end{tikzpicture}
\end{equation*}
\begin{center}
    \textbf{Figure 6.1:} The fan of $W$.
\end{center}
which includes the cones $\sigma_{13},\sigma_{23}$ bounded by the rays generated by the $\{v_i\}$.  According to the orbit-cone correspondence, denote by $D_i\subset W$ the torus-invariant divisor corresponding to $v_i$. Note that $W$ may be viewed as a resolution of the singular toric variety associated to the cone $\sigma_{12}$, with exceptional divisor $D_3$ of self-intersection $D_3\cdot D_3=-(a+b)$.

Writing $x,y$ for the toric coordinates on $W\setminus \cup D_i=(\mathbb{C}^*)^2$, we have divisors $(x)=aD_1+bD_2+D_3$ and $(y)=D_1-D_2$.  Consider the maps from a punctured disk to $(\mathbb{C}^*)^2$ given by $\mu_1\colon t\mapsto (t^a,t)$, $\mu_2\colon t\mapsto (t^b,t^{-1})$, and $\mu_3\colon t\mapsto (t,1)$.  Evidently these extend to maps from the full disk to $W$, with images $C_i$ satisfying $C_i\cdot D_j=\delta_{ij}$.  It follows that loops $\gamma_i\in \pi_1((\mathbb{C}^*)^2)$ about $D_1$, $D_2$ and $D_3$ respectively are given by $\theta\mapsto (e^{2\pi i\theta},e^{2\pi i\theta})$, $(e^{2\pi i b \theta},e^{-2\pi i \theta})$, and $(e^{2\pi i \theta},1)$.  Let $T_i$ denote their images under some (e.g. monodromy) representation.

\begin{sublemma}\label{sl6.1}
We have $T_3^{a+b}=T_1T_2$.
\end{sublemma}

\begin{proof}
The image of $\theta\mapsto (e^{2\pi i m_1\theta},e^{2\pi i m_2\theta})$ in $\pi_1((\mathbb{C}^*)^2)\cong \mathbb{Z}^2$ is $(m_1,m_2)$.  So $\gamma_3^{a+b}$ maps to $(a+b)(1,0)=(a+b,0)$ and $\gamma_1\gamma_2$ maps to $(a,1)+(b,-1)=(a+b,0)$.
\end{proof}

The fifth tangent point is locally modelled by $y(y-x^5)=0$ which is a type $A_9$ singularity. Its resolution graph is given in \textbf{Figure 6.2}. We can apply a base change so that all monodromy operators become unipotent. By applying the \textbf{Sublemma \ref{sl6.1}} (in the cases $a+b=1$ or $2$) to the toric resolution diagram in \textbf{Figure 6.2}, we obtain the following monodromy relations:
\begin{equation}
T_1^2=T_2; \;T_2^2=T_1T_3; \;T^3_2=T_2T_4; \;T_4^2=T_3T_5; \;T_5=T_4T_{E_1}T_y
\end{equation}
Which implies all $T_i$ have the same unipotent monodromy logarithm.

\begin{equation*}
\begin{tikzpicture}[scale=2.5]
    \node at (-1.3,2.8) (T1) {$T_1$};
    \path[->] (T1) edge  [loop right] node {} ();
    \node at (-0.9,2.4) (T2) {$T_2$};
    \path[->] (T2) edge  [loop below] node {} ();
    \node at (-0.9,2.6) (T3) {$T_3$};
    \path[->] (T3) edge  [loop right] node {} ();
    \node at (-0.4,2.2) (T4) {$T_4$};
    \path[->] (T4) edge  [loop below] node {} ();
    \node at (-0.5,2.4) (T5) {$T_5$};
    \path[->] (T5) edge  [loop right] node {} ();
    \node at (0.1,1.6) (Ta) {$T_y$};
    \path[->] (Ta) edge  [loop above] node {} ();
    \node at (-0.1,1.7) (Tb) {$T_{E_1}$};
    \path[->] (Tb) edge  [loop below] node {} ();
\draw[-,line width=1.0pt] (-1,3) -- (-1,1.9); 
\draw[-,line width=1.0pt] (-1.1,1.9) -- (-0.5,2.5);
\draw[-,line width=1.0pt] (-0.6,2.8) -- (-0.6,1.6);
\draw[-,line width=1.0pt] (-0.7,1.6) -- (-0.1,2.2);
\draw[-,line width=1.0pt] (-0.2,2.6) -- (-0.2,1.3);
\draw[-,line width=1.0pt] (-0.3,1.5) -- (0.3,1.5);
\draw[-,line width=1.0pt] (-0.3,1.8) -- (0.3,1.8);
\end{tikzpicture}
\mspace{30mu}
\begin{tikzpicture}[scale=1]
\fill (0,3) circle (2pt);
\node at (-0.5,3) {$(0,1)$};
\draw[-,line width=0.3pt] (0,1) -- (0,3.5);
\draw[-,line width=0.3pt] (0,1) -- (5.5,1);
\draw [-stealth](0,1) -- (5,3) [line width = 1pt];
\draw [-stealth](0,1) -- (4,3) [line width = 1pt];
\draw [-stealth](0,1) -- (3,3) [line width = 1pt];
\draw [-stealth](0,1) -- (2,3) [line width = 1pt];
\draw [-stealth](0,1) -- (1,3) [line width = 1pt];
\draw [-stealth](0,1) -- (1,1) [line width = 1pt];
\node at (1,3.2) {$(1,1)$};
\node at (2,3.2) {$(2,1)$};
\node at (3,3.2) {$(3,1)$};
\node at (4,3.2) {$(4,1)$};
\node at (5,3.2) {$(5,1)$};
\node at (1,0.7) {$(1,0)$};
\end{tikzpicture}
\end{equation*}
\begin{center}
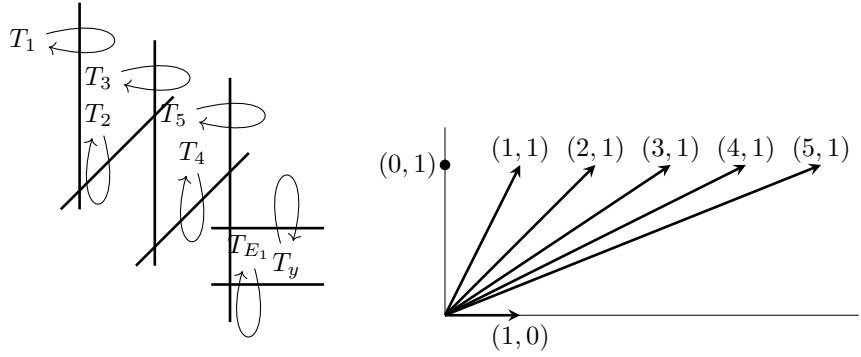
    \textbf{Figure 6.2:} Resolution Graph and Toric Resolution of $A_9$-singularity.
\end{center}

\subsection{Proof of Theorem 5.3}\label{S6.2}

According to \textbf{Theorem \ref{thm2.8}}, monodromy cones can be described by their associated Limiting Mixed Hodge Structures (LMHS), which should be considered modulo the actions of rescaling and by $\Gamma$. We refer readers to \cite{BPR16} and \cite{KPR19} for details on the classification of nilpotent orbits by their associated Limiting Mixed Hodge Structure types, as well as their polarized relations. In this paper we directly apply the results for weight-3 Calabi-Yau type (cf. \cite[Example 5.8]{KPR19}). 

We start by looking at the nilpotent cone types of $\sigma_i, i=1,2,3,4$ in the sense of \cite{KPR19}.  The notation $\langle A_1\mid B\mid A_2\rangle$ means that the LMHS type of $N_1$ and $N_2$ are $A_1$ and $A_2$, while that of $N\in int(\sigma)$ is $B$.  It is easy to see that $\sigma_3$ and $\sigma_4$ are both of type $\langle \mathrm{I}_1\mid\mathrm{I}_2\mid \mathrm{I}_1\rangle$.  (See \textbf{Figure 6.3} for the meaning of $\mathrm{I}_1$ etc.)  On the other hand, $\sigma_1$ is of type $\langle \mathrm{IV}_2\mid\mathrm{IV}_2\mid \mathrm{IV}_2\rangle$, while $\sigma_2$ has type $\langle \mathrm{I}_1\mid\mathrm{IV}_2\mid \mathrm{IV}_2\rangle$.  An important remark here is the LMHS type of a nilpotent cone does not need to agree with those from its faces; as we can see, cones $\sigma_2, \sigma_3$ and $\sigma_4$ are such examples.

To prove $\sigma_i$ and their $Ad(\Gamma)$-orbits give a polyhedral fan, we need to prove for any two cones lies in these orbits, the worst intersection we could expect is gluing along a boundary vector; that is, there are no two cones intersecting along an interior vector of both (which we will refer to as an "interior intersection" in the remaining of this section).

\begin{equation*}
\begin{tikzpicture}[scale=0.7]
    \draw[-,line width=1.0pt] (0,0) -- (0,3.5);
    \draw[-,line width=1.0pt] (0,0) -- (3.5,0);
\fill (3,3) circle (3pt);
\node at (2,-0.5) {};
\fill (0,0) circle (3pt);
\fill (1.2,0.8) circle (3pt);
\fill (0.8,1.2) circle (3pt);
\fill (1.8,2.2) circle (3pt);
\fill (2.2,1.8) circle (3pt);
\draw [-stealth](2.8,2.8) -- (2.2,2.2) [line width = 1.0pt];
\draw [-stealth](1.6,2.0) -- (1.0,1.4) [line width = 1.0pt];
\draw [-stealth](2.0,1.6) -- (1.4,1.0) [line width = 1.0pt];
\draw [-stealth](0.8,0.8) -- (0.2,0.2) [line width = 1.0pt];
\node at (0.8,2) {};
\node at (1.7,-1) {LMHS of Type $\mathrm{IV}_2$};
\end{tikzpicture}
\mspace{30mu}
\begin{tikzpicture}[scale=0.7]
    \draw[-,line width=1.0pt] (0,0) -- (0,3.5);
    \draw[-,line width=1.0pt] (0,0) -- (3.5,0);
\fill (0,3) circle (3pt);
\fill (3,0) circle (3pt);
\fill (1.2,0.8) circle (3pt);
\fill (0.8,1.2) circle (3pt);
\fill (1.8,2.2) circle (3pt);
\fill (2.2,1.8) circle (3pt);
\draw [-stealth](1.6,2.0) -- (1.0,1.4) [line width = 1.0pt];
\draw [-stealth](2.0,1.6) -- (1.4,1.0) [line width = 1.0pt];
\node at (0.8,2) {} ;
\node at (1.7,-1) {LMHS of Type $\mathrm{I}_2$};
\end{tikzpicture}
\mspace{30mu}
\begin{tikzpicture}[scale=0.7]
    \draw[-,line width=1.0pt] (0,0) -- (0,3.5);
    \draw[-,line width=1.0pt] (0,0) -- (3.5,0);
\fill (0,3) circle (3pt);
\fill (3,0) circle (3pt);
\fill (1,1) circle (3pt);
\fill (1,2) circle (3pt);
\fill (2,1) circle (3pt);
\fill (2,2) circle (3pt);
\draw [-stealth](1.8,1.8) -- (1.2,1.2) [line width = 1.0pt];
\node at (0.8,2){} ;
\node at (1.7,-1) {LMHS of Type $\mathrm{I}_1$};
\end{tikzpicture}
\end{equation*}
\begin{center}
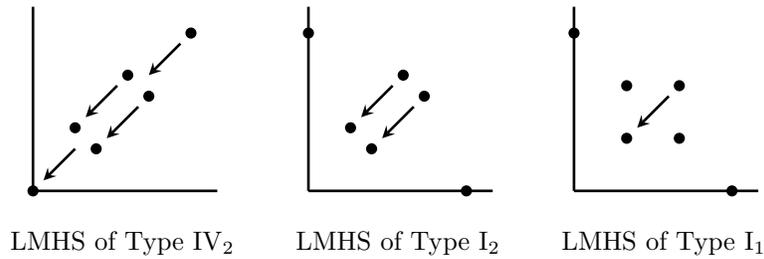
    \textbf{Figure 6.3:} Some LMHS types for $(1,2,2,1)$
\end{center}

Notice that conjugation by an element in $Sp(6,\mathbb{Z})$ does not change the type of a monodromy nilpotent vector defined by \cite{KPR19}, thus we only have the following possible interior intersections between different cones and their $Ad(\Gamma)$-orbits that might violate the definition of fan:

\begin{description}
\item[Case 1.] For $i,j\in\{1,2\}$, a cone in $\sigma_i$'s $Ad(\Gamma)$-orbit may intersect with $\sigma_j$;
\item[Case 2.] For $i,j\in\{3,4\}$, a cone in $\sigma_i$'s $Ad(\Gamma)$-orbit may intersect with $\sigma_j$;
\end{description}
We shall investigate these two cases separately.  Note that in principle we can take $\Gamma\leq  Sp(6,\mathbb{Z})$ to be any finite-index subgroup, so what really needs to be shown is that $Sp(6,\mathbb{Z})$-orbits do not create \emph{infinitely many} interior intersections in a single cone.

\subsubsection{Analysis of case 1}

In this case we proved a stronger result: $\sigma_i$ and $Ad(Sp(6,\mathbb{Z}))\sigma_j$ have no interior intersection for $i,j\in\{1,2\}$. Suppose not; i.e., there exists a $M \in Sp(6,\mathbb{Z})$ such that for $i,j\in \{1,2\}$, $\sigma_i$ and $M^{-1}\sigma_jM$ intersecting at some interior point. By \textbf{Theorem \ref{thm2.8}}, $\sigma_i$ and $M^{-1}\sigma_jM$ give the same monodromy weight filtration $W_{\bullet}$, which implies $M$ should lie in the corresponding parabolic subgroup preserving $W_{\bullet}$. Moreover, if we consider the image of $\sigma_i$ and $M^{-1}\sigma_jM$ in the graded pieces $Gr_{-2}\mathfrak{g}:= \{M \in \mathfrak{sp}(6,\mathbb{Q}) | MW_k \in W_{k-2}\}$ as well as $M$'s adjoint action on $Gr_{-2}$, the only effective part is $M$'s Levi factor.

To be precise, we use the same notation for cones, vectors in $\mathfrak{g}$ and their image in $Gr_{-2}\mathfrak{g}$. For example,

 \small
 \begin{equation}\label{eq6.6}
N_x=
\begin{pmatrix}
0 & -1 & 0 & 0 & 0 & 0\\
0 & 0 & 0 & -10 & -5 & 0\\
0 & 0 & 0 & -10 & -10 & 0\\
0 & 0 & 0 & 0 & 0 & 0\\
0 & 0 & 0 & 0 & 0 & 1\\
0 & 0 & 0 & 0 & 0 & 0
\end{pmatrix},
N_y=
\begin{pmatrix}
0 & 0 & -1 & 0 & 0 & 0\\
0 & 0 & 0 & -10 & -10 & 0\\
0 & 0 & 0 & -5 & -10 & 0\\
0 & 0 & 0 & 0 & 0 & 1\\
0 & 0 & 0 & 0 & 0 & 0\\
0 & 0 & 0 & 0 & 0 & 0
\end{pmatrix},
\end{equation}
\begin{equation}
N_{E_1}=
\begin{pmatrix}
0 & 0 & 0 & 0 & 0 & 0\\
0 & 0 & 0 & 0 & 50 & 0\\
0 & 0 & 0 & 0 & 0 & 0\\
0 & 0 & 0 & 0 & 0 & 0\\
0 & 0 & 0 & 0 & 0 & 0\\
0 & 0 & 0 & 0 & 0 & 0
\end{pmatrix}
\end{equation}
\normalsize
and the Levi factor of elements in the parabolic subgroup of $Sp(6,\mathbb{Z})$ defined by the weight filtration $W_{\bullet}$ (up to a sign) has the form:
\small
\begin{equation}
g=
\begin{pmatrix}
1 & 0 & 0 & 0 & 0 & 0\\
0 & a & b & 0 & 0 & 0\\
0 & c & d & 0 & 0 & 0\\
0 & 0 & 0 & a & -b & 0\\
0 & 0 & 0 & -c & d & 0\\
0 & 0 & 0 & 0 & 0 & 1
\end{pmatrix},
\begin{pmatrix}
a & b\\
c & d
\end{pmatrix} \in SL(2,\mathbb{Z}).
\end{equation}
\normalsize

Suppose there is a $g$ such that $g^{-1}\sigma_1g$ intersects with $\sigma_1$ at the half-line $sN_x+tN_y$ where $s,t > 0$. Let $N = N_{s,t} := sN_x+tN_y$ where $s,t$ are positive integers, then the triple $\{g^{-1}N_xg, g^{-1}N_yg, N\}$ lies in a plane, and so does the triple $\{N_xg, N_yg, gN\}$. Writing these relations in terms of $g$ with the above form, we have:

\begin{equation}
N_xg=
\begin{pmatrix}
 0 & -a & -b & 0 & 0 & 0\\ 
 0 & 0 & 0 & -10a+5c & 10b-5d & 0\\ 
 0 & 0 & 0 & -10a+10c & 10b-10d & 0\\
 0 & 0 & 0 & 0 & 0 & 0\\
 0 & 0 & 0 & 0 & 0 & 1\\
 0 & 0 & 0 & 0 & 0 & 0
\end{pmatrix}, 
\end{equation}
\begin{equation}
N_yg=
\begin{pmatrix}
 0 & -c & -d & 0 & 0 & 0\\ 
 0 & 0 & 0 & -10a+10c & 10b-10d & 0\\ 
 0 & 0 & 0 & -5a+10c & 5b-10d & 0\\
 0 & 0 & 0 & 0 & 0 & 1\\
 0 & 0 & 0 & 0 & 0 & 0\\
 0 & 0 & 0 & 0 & 0 & 0
\end{pmatrix},
\end{equation}\tiny
\begin{equation}
gN=
\begin{pmatrix}
 0 & -s & -t & 0 & 0 & 0\\ 
 0 & 0 & 0 & -10(a+b)s-(10a+5b)t & -(5a+10b)s-10(a+b)t & 0\\ 
 0 & 0 & 0 & -10(c+d)s-(10c+5d)t & -(5c+10d)s-10(c+d)t & 0\\
 0 & 0 & 0 & 0 & 0 & -bs+at\\
 0 & 0 & 0 & 0 & 0 & ds-ct\\
 0 & 0 & 0 & 0 & 0 & 0
\end{pmatrix}
\end{equation}
\normalsize
Therefore, by aligning non-zero entries, the following $8 \times 3$ matrix should be singular:

\begin{equation}
R:=
\begin{pmatrix}
a & c & s\\
b & d & t\\
10a-5c & 10a-10c & 10(a+b)s+(10a+5b)t \\
-10b+5d & -10b+10d & (5a+10b)s+10(a+b)t\\
10a-10c & 5a-10c & 10(c+d)s+(10c+5d)t\\
-10b+10d & -5b+10d & (5c+10d)s+10(c+d)t\\
0 & -1 & bs-at\\
-1 & 0 & -ds+ct
\end{pmatrix}.
\end{equation}
\normalsize
Write $R_{ijk}$ for the $3 \times 3$ submatrix obtained by taking the $i,j,k$th-columns of $R$, then: 

\begin{equation}
\det(R_{178}) =
\begin{vmatrix}
a & c & s\\
0 & -1 & bs-at\\
-1 & 0 & -ds+ct
\end{vmatrix} =2a(ds-ct),
\end{equation}
\begin{equation}
\det(R_{278}) =
\begin{vmatrix}
b & d & t\\
0 & -1 & bs-at\\
-1 & 0 & -ds+ct
\end{vmatrix} =2b(ds-ct).
\end{equation}
\normalsize
Since $a,b$ can not both vanish, we must have $ds-ct=0$. Without loss of generality suppose $s=c, t=d$.

Now $R$ becomes:

\begin{equation}
R=
\begin{pmatrix}
a & c & c\\
b & d & d\\
10a-5c & 10a-10c & 10(a+b)c+(10a+5b)d \\
-10b+5d & -10b+10d & (5a+10b)c+10(a+b)d\\
10a-10c & 5a-10c & 10(c+d)c+(10c+5d)d\\
-10b+10d & -5b+10d & (5c+10d)c+10(c+d)d\\
0 & -1 & -1\\
-1 & 0 & 0
\end{pmatrix}.
\end{equation}
\normalsize
It is clear that the only possibility for $R$ to be singular is when the second column and the third column are identical. By taking difference between $3^{th}$ and $4^{th}$ rows as well as $5^{th}$ and $6^{th}$ rows, we get the following equations:

\begin{equation} 
\begin{split}
&2(a+b-c-d) =  ac-bd \\ 
&a+b-2c-2d =  c^2-d^2 \\
&ad-bc=1
\end{split}
\end{equation}
We may view this as a linear system of the form $Ax=b$ where $x = [a, b]^t, b=[1, (c+1)^2-(d-1)^2, 2c+2d]^t$, and
\begin{equation}
A=
\begin{pmatrix}
d & -c \\
1 & 1 \\
2-c & 2+d
\end{pmatrix}.
\end{equation}
The rank of $A$ is $2$ if $c+d\neq 0$ and $1$ otherwise, and the determinant of the matrix $(A|b)$ is $(c+d)((c+d)((c-d)^2-2)+1)$. Therefore, the only chance where $Ax=b$ has a solution is $c+d\neq 0$ and 
$(c+d)((c-d)^2-2)=-1$.  The only 
\tiny
$\begin{pmatrix}
a & b\\
c & d
\end{pmatrix}$ \normalsize $ \in SL(2, \mathbb{Z})$ 
satisfying the above equation is the identity. Therefore, $\sigma_1$ has no non-trivial intersections with its own $Ad(Sp(6,\mathbb{Z}))$-orbit. Similarly we can prove the same statement holds for $\sigma_2$.

\bigskip

Now we shall check the intersection between $\sigma_1$ and $Ad(Sp(6,\mathbb{Z}))\sigma_2$. Any vector lying in the interior of $\sigma_2$ can be generated by some $N_t := N_y + tN_{E_1}$ with $t \in \mathbb{Q}_{> 0}$. In matrix form it is written as:
\begin{equation}
N_t=
\begin{pmatrix}
 0 & 0 & -1 & 0 & 0 & 0\\ 
 0 & 0 & 0 & -10 & -10+50t & 0\\ 
 0 & 0 & 0 & -5 & -10 & 0\\
 0 & 0 & 0 & 0 & 0 & 1\\
 0 & 0 & 0 & 0 & 0 & 0\\
 0 & 0 & 0 & 0 & 0 & 0
\end{pmatrix}.
\end{equation}
We investigate the column matrix generated by $N_xg, N_yg$ and $gN_t$ as we did before:
\begin{equation}
R':=
\begin{pmatrix}
a & c & 0\\
b & d & 1\\
10a-5c & 10a-10c & 10a+5b \\
-10b+5d & -10b+10d & 10a+10b-50at\\
10a-10c & 5a-10c & 10c+5d\\
-10b+10d & -5b+10d & 10c+10d-50ct\\
0 & -1 & -a\\
-1 & 0 & c
\end{pmatrix}.
\end{equation}
If $R'$ is singular, the first (or the last) two rows of $R'$ imply the linear relation $-cN_xg+aN_yg=gN_t$. Applying this relation to the third and fifth rows of $R'$, we get the relations: 
\begin{equation} 
\begin{split}
b = 2a^2+c^2-4ac-2a \\
d = a^2+2c^2-4ac-2c
\end{split}
\end{equation}
and hence 
\begin{equation}\label{eq6a}
\begin{vmatrix}
a & b\\
c & d
\end{vmatrix}=
\begin{vmatrix}
a & 2a^2+c^2-4ac-2a\\
c & a^2+2c^2-4ac-2c
\end{vmatrix}=(a-c)(a^2-5ac+c^2)=1
\end{equation}
The only 
$\begin{pmatrix}
a & b\\
c & d
\end{pmatrix}$ in $SL(2,\mathbb{Z})$ satisfying \eqref{eq6a} are:
$\begin{pmatrix}
1 & 0\\
0 & 1
\end{pmatrix}$ and
$\begin{pmatrix}
0 & 1\\
-1 & 0
\end{pmatrix}$,
but the second matrix gives $t=0$ which means the intersection is on the boundary.

\subsubsection{Analysis of case 2}

We first check possible interior intersections in $\sigma_3$'s $Ad(Sp(6,\mathbb{Z}))$-orbit. Recalling the matrix representation of $N_{p_1}$ and $N_{p_3}$ in \eqref{eq5.10}, we shall find a new basis under which $N_{p_1}$ and $N_{p_3}$ have the standard upper-triangular form. Notice that under such a rational basis which has the same intersection matrix as $Q$ in \eqref{eq5.3}, type \textbf{$I_1$} and type \textbf{$I_2$} monodromy nilpotent operators should lie in a $3$-dimensional Abelian Lie subalgebra of $\mathfrak{g}_{\mathbb{Q}}$ and have the standard form

\begin{equation}\label{eq6b}
N=
\begin{pmatrix}
 0 & 0 & 0 & 0 & x & y\\ 
 0 & 0 & 0 & 0 & z & x\\ 
 0 & 0 & 0 & 0 & 0 & 0\\
 0 & 0 & 0 & 0 & 0 & 0\\
 0 & 0 & 0 & 0 & 0 & 0\\
 0 & 0 & 0 & 0 & 0 & 0
\end{pmatrix}
\end{equation}
with $(x,y,z)\in \mathbb{Q}^3$. Now let
\begin{equation}
P_1:=
\begin{pmatrix}
 -5 & -5 & -1 & -1 & 1 & 0\\ 
 2 & 5 & 1 & -4 & -1 & 1\\ 
 5 & 2 & 1 & -1 & 0 & 0\\
 0 & -1 & 0 & 1 & 0 & 0\\
 -1 & 0 & 0 & 0 & 0 & 0\\
 -1 & -1 & 0 & 0 & 0 & 0
\end{pmatrix}, 
\end{equation}
We have $P_1\in Sp(6,\mathbb{Z})$, and
\begin{equation}
P_1^{-1}N_{p_1,1}P_1=
\begin{pmatrix}
 0 & 0 & 0 & 0 & 0 & 0\\ 
 0 & 0 & 0 & 0 & -1 & 0\\ 
 0 & 0 & 0 & 0 & 0 & 0\\
 0 & 0 & 0 & 0 & 0 & 0\\
 0 & 0 & 0 & 0 & 0 & 0\\
 0 & 0 & 0 & 0 & 0 & 0
\end{pmatrix}, 
\end{equation}
\begin{equation}
P_1^{-1}N_{p_1,2}P_1=
\begin{pmatrix}
 0 & 0 & 0 & 0 & 0 & -1\\ 
 0 & 0 & 0 & 0 & 0 & 0\\ 
 0 & 0 & 0 & 0 & 0 & 0\\
 0 & 0 & 0 & 0 & 0 & 0\\
 0 & 0 & 0 & 0 & 0 & 0\\
 0 & 0 & 0 & 0 & 0 & 0
\end{pmatrix} 
\end{equation}
which implies that if we identify \eqref{eq6b} with $(x,y,z)\in \mathbb{R}^3$, the cone $\sigma_0:=P_1^{-1}\sigma_3 P_1$ is identified with the cone $\mathbb{R}_{\geq 0}<(0,-1,0),(0,0,-1)>$ in $\mathbb{R}^3$. Under the new basis, elements in the corresponding parabolic subgroup of $Sp(6,\mathbb{Z})$ have their Levi factors of the form:
\begin{equation}
M_L=
\begin{pmatrix}
 \begin{matrix}
 a & b\\
 c & d
 \end{matrix} &  &  \\
  & A &  \\
  & & \begin{matrix}
 a & -b\\
 -c & d
 \end{matrix}
\end{pmatrix} 
\end{equation}
where \tiny
$\begin{pmatrix}
a & b\\
c & d
\end{pmatrix}$
\normalsize  
and $A$ are both in $SL(2,\mathbb{Z})$, and
\begin{equation}
M_L^{-1}{\sigma_0}M_L=
\begin{pmatrix}
 0 & 0 & 0 & 0 & cds+abt & -d^2s-b^2t\\ 
 0 & 0 & 0 & 0 & -c^2s-a^2t & cds+abt\\ 
 0 & 0 & 0 & 0 & 0 & 0\\
 0 & 0 & 0 & 0 & 0 & 0\\
 0 & 0 & 0 & 0 & 0 & 0\\
 0 & 0 & 0 & 0 & 0 & 0
\end{pmatrix} 
\end{equation}
where $s,t>0$. It is clear that $M_L^{-1}{\sigma_0}M_L$ and ${\sigma_0}$ have interior intersection if and only if $cds+abt=0$, which implies $ab$ and $cd$, hence $ad$ and $bc$ are both non-zero and have different signs, which is impossible for $a,b,c,d\in \mathbb{Z}, ad-bc=1$. Therefore, there is no self interior intersection in ${\sigma_0}$'s $Ad(Sp(6,\mathbb{Z}))$-orbit. Since $P_1\in Sp(6,\mathbb{Z})$, it follows that $\sigma_3$ has no interior intersection with its own $Ad(Sp(6,\mathbb{Z}))$-orbit.

\bigskip
The same idea can be applied to investigate $\sigma_4$. We have $P_3^{-1}\sigma_4P_3={\sigma_0}$, where 

\begin{equation}
P_3:=
\begin{pmatrix}
 40 & -40 & -\frac{3}{4} & -\frac{1}{4} & \frac{1}{16} & 0\\ 
 -4 & 20 & 1 & -\frac{4}{3} & -\frac{1}{12} & -\frac{1}{12}\\ 
 -20 & 4 & 1 & 0 & 0 & 0\\
 0 & -12 & 0 & 1 & 0 & 0\\
 12 & 0 & 0 & 0 & 0 & 0\\
 16 & -16 & 0 & 0 & 0 & 0
\end{pmatrix}
\end{equation}
is in $Sp(6,\mathbb{Q})$. Let
\begin{equation}\label{eq6c}
\Gamma_4:=P_3^{-1}Sp(6,\mathbb{Z})P_3\leq Sp(6,\mathbb{Q}),
\end{equation}
then $\sigma_4$'s $Ad(\Gamma_4)$-orbit has no self-interior intersections.  Intersecting \eqref{eq6c} with $Sp(6,\mathbb{Z})$, of course, yields a finite-index subgroup.

Finally, we check the interplay between $\sigma_3$ and $\sigma_4$. These two cones are conjugated under $Sp(6,\mathbb{Q})$, by the matrix:
\begin{equation}
P:= P_3P_1^{-1}=
\begin{pmatrix}
 \frac{1}{16} & 0 & -\frac{11}{16} & -\frac{7}{8} & -\frac{1327}{16} & \frac{627}{16}\\ 
 -\frac{1}{6} & -\frac{1}{12} & \frac{11}{12} & -\frac{11}{12} & \frac{313}{12} & -\frac{101}{6}\\ 
 0 & 0 & 1 & 1 & 28 & -3\\
 0 & 0 & 0 & 1 & -11 & 11\\
 0 & 0 & 0 & 0 & -12 & 0\\
 0 & 0 & 0 & 0 & -32 & 16
\end{pmatrix}, 
\end{equation}
which lies in $Sp(6,\mathbb{Q})$. As $\sigma_3=P^{-1}\sigma_4P$, suppose there exists $M\in Sp(6,\mathbb{Z})$ such that $M^{-1}\sigma_4M$ and $\sigma_3=P^{-1}\sigma_4P=P_1P_3^{-1}\sigma_4P_3P_1^{-1}$ have interior intersections. Since $P_1^{-1}\sigma_3P_1=P_3^{-1}\sigma_4P_3={\sigma_0}$, it is equivalent to say ${\sigma_0}$ and $(M')^{-1}{\sigma_0}M'$ have interior intersections, where
\begin{equation}
M':=P_1^{-1}M^{-1}P_3
\end{equation}
is an element in $Sp(6,\mathbb{Q})$. It is clear by the previous argument, a contradiction will be caused if $M'\in Sp(6,\mathbb{Z})$.

For a natural number $n$, denote by $\Gamma(n)\leq Sp(6,\mathbb{Z})$  the kernel of the natural map $Sp(6,\mathbb{Z})\rightarrow Sp(6,\mathbb{Z}/n\mathbb{Z})$. The congruence subgroup $\Gamma(n)$ is of finite index in  $Sp(6,\mathbb{Z})$, as $Sp(6,\mathbb{Z}/n\mathbb{Z})$ is a finite group.

Now we take $n=48^2$ and let $\Gamma:=\Gamma(n)$.  We have $[Sp(6,\mathbb{Z}): \Gamma]<\infty$, and it can be directly checked $P_i^{-1}\Gamma P_i\leq Sp(6,\mathbb{Z}), i=1, 3$, and $P_1^{-1}\Gamma P_3\subset Sp(6,\mathbb{Z})$. Therefore, $\sigma_3, \sigma_4$ and their $Ad(\Gamma)$-orbits have no self interior intersections. The proof of \textbf{Theorem \ref{thm5.3}} is now complete.

\begin{remark}
 We are expecting the intersecting behaviors between nilpotent cones as well as their orbits under the adjoint action by some arithmetic group $\Gamma$ an essential characteristic that only depends on the nilpotent cone types introduced at the beginning of this section, as nilpotent cones and their orbits under $Ad(\Gamma)$-action are essentially classified in this way.
\end{remark}

\section{Some Mumford-Tate Domain aspects}\label{S7}

The construction of Kato-Usui spaces as well as examples shown in the previous sections are connected with the theory of Mumford-Tate domains, which is a concept generalizing that of period domains. We shall briefly review some basic results about this topic, and then turn to the proof of \textbf{Theorem \ref{thm5.4}}. 

Regarding the general theory of Mumford-Tate domains, we refer to \cite{GGK12} for the fundamental matrerial, and \cite{Moo04} as good introductory notes. The compatibility of Kato-Usui's construction with the Mumford-Tate domain setting is summarized in detail by M. Kerr and G. Pearlstein, see \cite[\S 6]{KP16}.

\medskip

We first recall some basic concepts. Given a Polarized Hodge Structure $\varphi: U(1)\rightarrow SL(V)$, the \textbf{Mumford-Tate group} $M_{\varphi}$ is the smallest $\mathbb{Q}$-algebraic subgroup of $SL(V)$ containing the image of $\varphi$. The \textbf{Mumford-Tate domain} $D_{\varphi}\subset D$ is defined to be the orbit $M_{\varphi,\mathbb{R}}\,\varphi$.  It is well-known that if we define the \textbf{Hodge tensors} of a Hodge structure by
\begin{equation}
    Hg_{\varphi}^{\bullet, \bullet}:= \bigcup_{p\in \mathbb{Z}}(T^{k,l}V\bigcap F_{k,l}^{\frac{n(k-l)}{2}}),
\end{equation}
where
\begin{equation}
    (T^{k,l}V):=V^{\bigotimes k}\otimes {V^*}^{\bigotimes l}
\end{equation}
inherits a natural weight $n(k-l)$ polarized Hodge structure $F_{k,l}^{\bullet}$ from $\varphi$, then the Mumford-Tate group is exactly the $\mathbb{Q}$-subgroup fixing all Hodge tensors.

\begin{definition}\label{def7.1}
We say a Hodge structure $\varphi$ is (Mumford-Tate) \textbf{generic} (resp. \textbf{CM}) if its associated Mumford-Tate group is $Aut(V,Q)$ (resp. a torus in $Aut(V,Q)$), or equivalently, its associated Mumford-Tate domain is the whole period domain $D$ (resp. a single point). 
\end{definition}

Besides these two ``end points'', Mumford-Tate domains are subdomains of the corresponding period domain. It has been proven that Mumford-Tate groups are essentially classified by Hodge representations, see \cite[Chap. IV.A]{GGK12}. Regarding the classification of Hodge representations of weight $\leq 3$, and hence the classification of possible Mumford-Tate groups of corresponding Hodge type, see Han and Robles' recent work \cite{HR20}.

For a geometric variation of Hodge structure given by a family, it is natural to ask how generic the family is in this sense.  Any period map into $D$ factors through the MT-domain associated to a ``Hodge generic'' point of the image.  The generic information about a period map can be analyzed from monodromy:

\begin{remark}
 For a period map $S\rightarrow \Gamma\backslash D$ where $\Gamma$ is the monodromy group, the (identity component of the) $\mathbb{Q}$-Zariski closure of $\Gamma$ is called the \textbf{geometric monodromy group}, denoted as $\bar{\Gamma}^{\mathbb{Q}}$. Denoting the generic Mumford-Tate group by $M$, we have $\bar{\Gamma}^{\mathbb{Q}} \trianglelefteq M^{Der}$ where the upperscript means the derived subgroup. For more details, we refer to \cite[Chap. IV.A]{GGK12} and its references.

\end{remark}

We now investigate \textbf{Theorem \ref{thm5.4}}. The idea of the proof is analyzing all possibilities of the generic Mumford-Tate domain associated to the family, and rulling out all but one of them based on monodromy. By \cite[Thm. IV.A.2]{GGK12}, Any (real) MT-group must contain a compact maximal torus. According to Han and Robles' classification of type $(1,2,2,1)$ Hodge representations in \cite{HR20}, all connected reductive subgroups $M\leq Sp(6,\mathbb{Q})$ that could be a MT-group for some elements in $D$ are (or are contained in):
\begin{description}
\item[1. ] A compact maximal torus, ruled out by any non-trivial monodromies among \eqref{eq5.8}-\eqref{eq5.11};
\item[2. ] $SL(2)\times SL(2)$, which corresponds to the tensor construction of Hodge types $(1,1)\otimes (1,1,1)$, ruled out by $\mathrm{I}_1$ and $\mathrm{I}_2$ degenerations (e.g. in $\sigma_3,\sigma_4$), since these are not tensor products of degenerations of $(1,1)$ and $(1,1,1)$;
\item[3. ] $U(2,1)$, which corresponds to a decomposition of the VHS (over an imaginary quadratic field) into $W\oplus \bar{W}$, with $W$ of type $(1,1,1,0)$ or $(1,2,0,0)$, ruled out by the Hodge-Tate type monodromies \eqref{eq5.8};
\item[4. ] $SL(2)\times Sp(4)$, which corresponds to the direct sum construction of Hodge types $(1,1,1,1)\oplus (1,1)$; 
\item[5. ] $Sp(6)$.
\end{description}
\medskip

The cases left are (4) and (5). We rule out (4) as follows. For a nilpotent $N\in \mathfrak{sp}(6,\mathbb{Q})$ and an element $v\in H_{\mathbb{C}}$ which does not lie in the image of $N$, we define an $l$-string of $N$ to be the collection of non-trivial elements $\{v, Nv,...,N^{l-1}v\}$ where $Nv\neq 0$ but $N^{l}v=0$. For example, according to the classification in \textbf{Figure 6.3}, both type $I_1$ and $I_2$ monodromies admit only $1$-strings while type $IV_2$ monodromies also admit $3$-strings.

It is clear in case (4), the only allowed strings are $1$-strings and $3$-strings. However, it can be checked both $N_x-N_y$ and $[(-5)\pm \sqrt{21}]N_x+2N_y$ admits $2$-strings, therefore (4) is ruled out and case (5) becomes the only survivor, which implies the image of this period map is generic.

\medskip

Another comment is for Hodge type $(1,2,2,1)$, according to Kerr-Pearlstein-Robles' classification as introduced at the beginning of \textbf{Section \ref{S6.2}}, The only LMHS type allows the existence of $2$-strings is type $\mathrm{III}_0$. Its Hodge-Deligne diagram is shown in \textbf{Figure 7.1} below.  We are not saying that we have this LMHS type here, only that we have ``a nilpotent element of this type'' generated by sums of logarithms of unipotent elements of $\Gamma$.  (This still rules out (4).)
\medskip

We conclude this section with several remarks: The generic properties of the family show parts of the Kato-Usui boundary components introduced in \textbf{Section \ref{S3}} we borrow from the period domain to compactify a given geometric variation are also ``generic'', which means that (considered together) they do not attach to any proper Mumford-Tate subdomains. In general, associated to any Mumford-Tate domain are boundary components given by nilpotent cones fixed by the associated Mumford-Tate group, as monodromy nilpotent elements are naturally Hodge $(-1,-1)$-tensors.  

The geometry of Kato-Usui boundary components can be studied by fibering them over the MT-domains associated to the $Gr^W$ of the LMHS; the points in the fibers record extension classes for distinct weight-graded pieces.  The overall structure is determined by the LMHS types of the nilpotent cones.  For the general theory, we refer to \cite{KP16} and \cite{GGR21}, and \cite{GGK08} for an extended example related to mirror quintic type Calabi-Yau threefolds (i.e. those with Hodge type (1,1,1,1)).

\begin{align*}
\begin{tikzpicture}[scale=0.7]
    \draw[-,line width=1.0pt] (0,0) -- (0,3.5);
    \draw[-,line width=1.0pt] (0,0) -- (3.5,0);
\fill (2,3) circle (3pt);
\fill (3,2) circle (3pt);
\fill (2,1) circle (3pt);
\fill (1,0) circle (3pt);
\fill (0,1) circle (3pt);
\fill (1,2) circle (3pt);
\draw [-stealth](2.8, 1.8) -- (2.2,1.2) [line width = 1.0pt];
\draw [-stealth](1.8, 0.8) -- (1.2,0.2) [line width = 1.0pt];
\draw [-stealth](1.8, 2.8) -- (1.2,2.2) [line width = 1.0pt];
\draw [-stealth](0.8, 1.8) -- (0.2,1.2) [line width = 1.0pt];
\node at (1.7,-1) {\textbf{Figure 7.1} LMHS of Type $III_0$};
\end{tikzpicture}
\end{align*}

\section{Appendix: Two examples regarding the non-existence of complete fan}

In this section, we provide two examples on how the procedures introduced in \textbf{Section \ref{S4}} can be used to construct nilpotent orbits that don't come from a given rational polyhedral fan $\Sigma\subset \mathfrak{g}_{\mathbb{Q}}$ generated by nilpotent elements.

\subsection{Watanabe's example of type $(2,2,2)$}

For the first example, we re-investigate the family provided in \cite{Wat08}, but via a rather different method.

Let $(h^{p,q})=(h^{2,0}, h^{1,1}, h^{0,2}) = (2,2,2)$. Suppose $H_{\mathbb{Z}}$ admits a free basis $\{e_i\}_{1\leq i\leq 6}$ and the intersection form $H_{\mathbb{Q}}\times H_{\mathbb{Q}} \rightarrow \mathbb{Q}$ given by
\begin{equation}
    (<e_i,e_j>)_{1\leq i,j\leq 6} = 
    \begin{pmatrix}
    -I_2 & 0 & 0\\
    0 & E & 0\\
    0 & 0 & E
    \end{pmatrix},
    I_2 = 
    \begin{pmatrix}
    1 & 0 \\
    0 & 1    
    \end{pmatrix},
    E = 
    \begin{pmatrix}
    0 & 1\\
    1 & 0
    \end{pmatrix}.
\end{equation}
Let $\sigma_0 = \mathbb{Q}_{\geq 0}<N_1, N_2>$, where
\begin{equation}
    N_1 = 
    \begin{pmatrix}
    0 & 0 & 0 & 0 & 0 & 0\\
    0 & 0 & 0 & 0 & 0 & 0\\
    0 & 0 & 0 & 0 & 0 & -1\\
    0 & 0 & 0 & 0 & 0 & 0\\
    0 & 0 & 0 & 1 & 0 & 0\\
    0 & 0 & 0 & 0 & 0 & 0
    \end{pmatrix},
    N_2 = 
    \begin{pmatrix}
    0 & 0 & 0 & 0 & 0 & 0\\
    0 & 0 & 0 & 0 & 0 & 0\\
    0 & 0 & 0 & 0 & 0 & 0\\
    0 & 0 & 0 & 0 & 0 & 1\\
    0 & 0 & -1 & 0 & 0 & 0\\
    0 & 0 & 0 & 0 & 0 & 0
    \end{pmatrix}.
\end{equation}
Let $F^{\bullet}$ be defined by $F^2 = \mathbb{C}(ie_1+e_2) \bigoplus \mathbb{C}e_6, F^1 = (F^2)^{\perp}$. According to \cite{Wat08}, $exp(\sigma_{0, \mathbb{C}})F^{\bullet}$ is a $\sigma_0$-nilpotent orbit. Let $N_0 = N_1+N_2 \in int(\sigma_0)$, and denote $G=Sp(6, \mathbb{Q}), \mathfrak{g}=Lie(G)\leq End(H_{\mathbb{Q}})$ with the induced filtration $\tilde{F}^{\bullet}$. We also write $C(N)$ or $C(\sigma)$ for the corresponding commutators in $\mathfrak{g}_{\mathbb{C}}$, and one computes that

\begin{equation}
    C(\sigma_0)\cap \tilde{F}^{-1}= \left\{
    \begin{pmatrix}
    0 & a_{12} & 0 & 0 & 0 & a_{16}\\
    -a_{12} & 0 & 0 & 0 & 0 & a_{26}\\
    0 & 0 & 0 & 0 & 0 & a_{36}\\
    0 & 0 & 0 & 0 & 0 & a_{46}\\
    a_{16} & a_{26} & -a_{46} & -a_{36} & 0 & 0\\
    0 & 0 & 0 & 0 & 0 & 0
    \end{pmatrix} |  a_{ij} \in \mathbb{C} \right\}
    ,
\end{equation}
\begin{equation}
    C(N_0)\cap \tilde{F}^{-1}= \left\{
    \begin{pmatrix}
    0 & a_{12} & a_{13} & a_{13} & 0 & a_{16}\\
    -a_{12} & 0 & a_{23} & a_{23} & 0 & a_{26}\\
   a_{13} & a_{23} & 0 & 0 & 0 & a_{36}\\
    a_{13} & a_{23} & 0 & 0 & 0 & a_{46}\\
    a_{16} & a_{26} & -a_{46} & -a_{36} & 0 & 0\\
    0 & 0 & 0 & 0 & 0 & 0
    \end{pmatrix} | a_{ij} \in \mathbb{C} \right\}.
\end{equation}
It is evident that $C(\sigma_0)\cap \tilde{F}^{-1}$ and $C(N_0)\cap \tilde{F}^{-1}$ are not equal. For example, we can take:\footnote{In this section, we use ``$N$'' to denote the ``$M$'' from the proof of Lemma \ref{lem4.4}, since we will choose it to be nilpotent.  This is not to be confused with elements in $\sigma$.}

\begin{equation}
N=
 \begin{pmatrix}
    0 & 2 & 1 & 1 & 0 & 0\\
    -2 & 0 & 1 & 1 & 0 & 0\\
    1 & 1 & 0 & 0 & 0 & 0\\
    1 & 1 & 0 & 0 & 0 & 0\\
    0 & 0 & 0 & 0 & 0 & 0\\
    0 & 0 & 0 & 0 & 0 & 0
\end{pmatrix},\;\;\;
exp(N) = 
 \begin{pmatrix}
    0 & 3 & 2 & 2 & 0 & 0\\
    -1 & 0 & 0 & 0 & 0 & 0\\
    0 & 2 & 2 & 1 & 0 & 0\\
    0 & 2 & 1 & 2 & 0 & 0\\
    0 & 0 & 0 & 0 & 1 & 0\\
    0 & 0 & 0 & 0 & 0 & 1
\end{pmatrix},
\end{equation}
and $F_{N}^{\bullet} := exp(N)F^{\bullet}$ can be represented as:
\begin{equation}\label{eq8.6}
    F_{N}^{2}=\mathbb{C}(3e_1-ie_2+2e_3+2e_4) \bigoplus \mathbb{C}e_6, F_{N}^{1} = (F_{N}^{2})^{\perp}.
\end{equation}
Now clearly we have $N\in C(N_0)\cap \tilde{F}^{-1} \subset C^h(N_0)$. However, it can be checked that $[N,N_1]\notin \tilde{F}^{-1}$, therefore it satisfies the assumption in \textbf{Lemma \ref{lem4.4}}.
(Also note that by \cite[Lemma 7.2]{Wat08} we see that $exp(\mathbb{C}N_0)F_{N}^{\bullet}$ is an $\mathbb{R}_{\geq 0}\langle N_0\rangle$-nilpotent orbit, but $exp(\mathbb{C}\sigma_0)F_{N}^{\bullet}$ is not a $\sigma_0$-nilpotent orbit as the Griffiths transversality fails.) We conclude by \textbf{Theorem \ref{thm4.6}} that there is no complete fan for $D$ of Hodge type $(2,2,2)$.

\medskip

There is another important observation. Denote $\tilde{F}_N^{\bullet}$ as the induced filtration by $F_N^{\bullet}$ on $\mathfrak{g}_{\mathbb{C}}$. Let
\begin{equation}
N':=
 \begin{pmatrix}
    0 & 2 & 1 & 1 & 0 & 0\\
    -2 & 0 & 1 & 1 & 0 & 0\\
    1 & 1 & 0 & 0 & 0 & 1\\
    1 & 1 & 0 & 0 & 0 & -1\\
    0 & 0 & 1 & -1 & 0 & 0\\
    0 & 0 & 0 & 0 & 0 & 0
\end{pmatrix},
\end{equation}
By computation we have $N'^3=0$ and $N', N_0\in \tilde{F}_N^{-1}$, as well as $[N', N_0]=0$. Moreover we have:
\begin{equation}
exp(iyN')=
\begin{pmatrix}
y^2+1 & -y^2+2yi & -y^2+yi & -y^2+yi & 0 & 0\\
-y^2-2yi & y^2+1 & y^2+yi & y^2+yi & 0 & 0\\
y^2+yi & -y^2+yi & -y^2+1 & -y^2 & 0 & -iy\\
y^2+yi & -y^2+yi & -y^2 & -y^2+1 & 0 & iy\\
0 & 0 & -iy & iy & 1 & 0\\
0 & 0 & 0 & 0 & 0 & 1
\end{pmatrix}, y\in \mathbb{R}
\end{equation}

If we use the basis vectors of $F_N^{2}$ in \eqref{eq8.6}, we will have: 
\begin{equation}
    i^2Q(exp(iyN')F_N^{2}, exp(-iyN')\overline{F_N^{2}})=
\begin{pmatrix}
2(2y+1)^2 & 0\\
0 & 4y^2
\end{pmatrix}
\end{equation}
which implies the positivity. Therefore, $N'$ and $N_0$ generate another $2$-dimensional nilpotent orbit with base point $F_N^{\bullet}$, and we have constructed a pair of nilpotent cones with maximal dimension intersecting along an interior point of both.
\medskip

We have the following generalization: let $N(a,b):=aN_1+bN_2$ ($a, b\in \mathbb{Z}_{>0}$), and set
\begin{equation}
\begin{split}
N_{a,b}&:=
 \begin{pmatrix}
    0 & 2\sqrt{ab} & b & a & 0 & 0\\
    -2\sqrt{ab} & 0 & b & a & 0 & 0\\
    a & a & 0 & 0 & 0 & 0\\
    b & b & 0 & 0 & 0 & 0\\
    0 & 0 & 0 & 0 & 0 & 0\\
    0 & 0 & 0 & 0 & 0 & 0
\end{pmatrix},\\
N_{a,b}&':= 
 \begin{pmatrix}
    0 & 2\sqrt{ab} & b & a & 0 & 0\\
    -2\sqrt{ab} & 0 & b & a & 0 & 0\\
    a & a & 0 & 0 & 0 & a\\
    b & b & 0 & 0 & 0 & -b\\
    0 & 0 & b & -a & 0 & 0\\
    0 & 0 & 0 & 0 & 0 & 0
\end{pmatrix}
\end{split}
\end{equation}
and $F_{N_{a,b}}^{\bullet}:=exp(N_{a,b})F^{\bullet}$. As above, we can prove $N(a,b)$ and $N_{a,b}'$ generate a nilpotent orbit with base point $F_{N_{a,b}}^{\bullet}$. Therefore, whenever $ab$ is a square number, there exists a nilpotent orbit with underlying rational nilpotent cone in $\mathfrak{g}_{\mathbb{Q}}$ intersecting with $\sigma_0$ at an interior point of both, and there are infinitely many such pairs $(a,b)$.

\bigskip
\subsection{Hosono and Takagi's example of type $(1,2,2,1)$}

Our second example is related to the family introduced in \textbf{Section \ref{S5}}, of Hodge type $h=(1,2,2,1)$. In this case, at a point in $D$, the induced Hodge structure on $\mathfrak{g}=End(H_{\mathbb{Q}})$ has weight $0$ and type $(h^{-3,3},...,h^{3,-3})=(1,2,5,5,5,2,1)$. Recall that one of the monodromy cones is generated by $N_x, N_y$ from \eqref{eq6.6}, which we reproduce for convenience:
 \begin{equation}\label{eq8.11}
N_x=
\begin{pmatrix}
0 & -1 & 0 & 0 & 0 & 0\\
0 & 0 & 0 & -10 & -5 & 0\\
0 & 0 & 0 & -10 & -10 & 0\\
0 & 0 & 0 & 0 & 0 & 0\\
0 & 0 & 0 & 0 & 0 & 1\\
0 & 0 & 0 & 0 & 0 & 0
\end{pmatrix},
N_y=
\begin{pmatrix}
0 & 0 & -1 & 0 & 0 & 0\\
0 & 0 & 0 & -10 & -10 & 0\\
0 & 0 & 0 & -5 & -10 & 0\\
0 & 0 & 0 & 0 & 0 & 1\\
0 & 0 & 0 & 0 & 0 & 0\\
0 & 0 & 0 & 0 & 0 & 0
\end{pmatrix}
\end{equation}
Writing $W_{\bullet}=W(N_x+N_y)_{\bullet}$, the elements 
\begin{equation}
\Upsilon:=\left\{
\begin{pmatrix}
0 & p & q & 0 & 0 & 0\\
0 & 0 & 0 & x & y & 0\\
0 & 0 & 0 & z & x & 0\\
0 & 0 & 0 & 0 & 0 & -q\\
0 & 0 & 0 & 0 & 0 & -p\\
0 & 0 & 0 & 0 & 0 & 0
\end{pmatrix}| p,q,x,y,z \in \mathbb{Q}\right\}
\end{equation}
of $\mathfrak{g}_{\mathbb{Q}}$ span the weight graded piece $Gr^W_{-2}\mathfrak{g}_{\mathbb{Q}}$.  Their intersections with the commutators $C(N_x), C(N_y)$ and $C(N_x+N_y)$ are given by

\begin{equation}
\begin{split}
    &C(N_x)\cap \Upsilon= \{(p,q,x,y,z) \in \mathbb{Q}^5 | x=10p+10q, y=5p+10q\} \\
    &C(N_y)\cap \Upsilon= \{(p,q,x,y,z) \in \mathbb{Q}^5 | x=10p+10q, z=10p+5q\} \\
    &C(N_x+N_y)\cap \Upsilon = \{(p,q,x,y,z) \in \mathbb{Q}^5 | x+z=20p+15q, x+y=15p+20q\}.
\end{split}
\end{equation}
Therefore we can take $N:=(1,1,22,13,13)$ to get a rational element in $[C(N_x+N_y)\backslash (C(N_x)\cap C(N_y))]\cap W_{-2}\mathfrak{g}_{\mathbb{Q}}$.  More precisely, 
\begin{equation}\label{eq8.14}
N=
\begin{pmatrix}
0 & -1 & -1 & 0 & 0 & 0\\
0 & 0 & 0 & -22 & -13 & 0\\
0 & 0 & 0 & -13 & -22 & 0\\
0 & 0 & 0 & 0 & 0 & 1\\
0 & 0 & 0 & 0 & 0 & 1\\
0 & 0 & 0 & 0 & 0 & 0
\end{pmatrix}, 
\end{equation}
\begin{equation}
exp(N)=
\begin{pmatrix}
1 & -1 & -1 & 35/2 & 35/2 & 35/3\\
0 & 1 & 0 & -22 & -13 & -35/2\\
0 & 0 & 1 & -13 & -22 & -35/2\\
0 & 0 & 0 & 1 & 0 & 1\\
0 & 0 & 0 & 0 & 1 & 1\\
0 & 0 & 0 & 0 & 0 & 1
\end{pmatrix}
\end{equation}
\normalsize

Next we are going to construct a $F^{\bullet}\in D$ which is polarized by $N_x$ and $N_y$. We recall more details from \cite{HT18} about the monodromy weight filtration. the canonical ordered basis for $H_{\mathbb{Q}}$ is:
\begin{equation}
    \{\alpha_0, \alpha_1, \alpha_2, \beta_2, \beta_1, \beta_0\}
\end{equation}
with the intersection form $Q$ in \eqref{eq5.3}, and the monodromy weight filtration $W_{\bullet}:=W_0\subset W_2\subset W_4\subset W_6 = H_{\mathbb{Q}}$ is given by:
\begin{equation}
    \begin{split}
        &W_0=<\alpha_0>, \\
        &W_2=<\alpha_0, \alpha_1, \alpha_2>,\\
        &W_4=<\alpha_0, \alpha_1, \alpha_2, \beta_2, \beta_1>,\\
        &W_6=<\alpha_0, \alpha_1, \alpha_2, \beta_2, \beta_1, \beta_0>,
    \end{split}
\end{equation}
We define $F^{\bullet}\in \check{D}$ as follows. Let $a_0=a+bi$, and
\begin{equation}\label{eq8.18}
    \begin{split}
        &F^3=<a_0\alpha_0+\beta_0>, \\
        &F^2=<a_0\alpha_0+\beta_0, N_x(a_0\alpha_0+\beta_0), N_y(a_0\alpha_0+\beta_0)>,\\
        &F^1=(F^3)^{\perp}
    \end{split}
\end{equation}
Now write $\sigma=\mathbb{R}_{>0}\langle N_x,N_y\rangle$.  To check whether $exp(\sigma_{\mathbb{C}})F^{\bullet}$ is a $\sigma$-nilpotent orbit, we need to check \textbf{(i)-(iii)} in \textbf{Definition \ref{def2.6}}. By the construction of $F^{\bullet}$, the commutativity and the Griffiths transversality are clear. To check the positivity, for $h\in \mathbb{R}_{>0}$ and $v\in F^3$, we need to look at the following $\mathbb{C}$-polynomials:
\begin{equation}
\begin{split}
    f_{N_x,v}(h):=i^{3-0}Q(exp(ihN_x)v, exp(-ihN_x)\bar{v})\\
    f_{N_y,v}(h):=i^{3-0}Q(exp(ihN_y)v, exp(-ihN_y)\bar{v})
\end{split}
\end{equation}
By computation, both polynomials are of degree $3$ and the coefficient of $h^3$-term in either polynomial is positive. Hence $exp(\sigma_{\mathbb{C}})F^{\bullet}$ is a $\sigma$-nilpotent orbit.  (Indeed, it corresponds to a Hodge-Tate degeneration.)  Moreover, one easily checks that $[N,N_x]F^3\not\subset F^1$, and so $N\notin C^h(\sigma)$ (while of course $N\in C(N_x+N_y)\subset C^h(N_x+N_y)$).  At this point, we can take $N_{\vec{v}}=N_x+N_y$ in the hypothesis of \textbf{Lemma \ref{lem4.4}} and apply \textbf{Theorem \ref{thm4.6}} to deduce that there is no complete fan.

\bigskip
Though strictly speaking unnecessary, it is instructive to investigate $F_N^{\bullet}:=exp(N)F^{\bullet}$ (with $N$ defined in \eqref{eq8.14}) a bit further to see the failure of $e^{\sigma_{\mathbb{C}}}F_N^{\bullet}$ to be a $\sigma$-nilpotent orbit quite explicitly, among other properties.  We have:
\begin{equation}\label{eq8.20}
    \begin{split}
        &F_N^3=<(a_0+\frac{35}{3})\alpha_0+\beta_0>, \\
        &F_N^2=<F_N^3, exp(N)N_x(a_0\alpha_0+\beta_0), exp(N)N_y(a_0\alpha_0+\beta_0)>,\\
        &F_N^1=(F_N^3)^{\perp}
    \end{split}
\end{equation}
Now we have:
\begin{equation}
N_x+N_y=
\begin{pmatrix}
0 & -1 & -1 & 0 & 0 & 0\\
0 & 0 & 0 & -20 & -15 & 0\\
0 & 0 & 0 & -15 & -20 & 0\\
0 & 0 & 0 & 0 & 0 & 1\\
0 & 0 & 0 & 0 & 0 & 1\\
0 & 0 & 0 & 0 & 0 & 0
\end{pmatrix}, 
\end{equation}
By the choice of $N$, $exp(N)$ and $N_x+N_y$ are commutative; hence
\begin{equation}
\begin{split}
(N_x+N_y)F_N^3&=(N_x+N_y)exp(N)F^3\\
              &=exp(N)(N_x+N_y)F^3\\
              &=exp(N)N_xF^3+exp(N)N_yF^3 \subset F_N^2,
\end{split}
\end{equation}
and 
$(N_x+N_y)F_N^2\subset F_N^1$ is clear. The positivity condition holds since
\begin{equation}
    i^{3-0}Q(e^{ih(N_x+N_y)}v, e^{-ih(N_x+N_y)}\overline{v})=\frac{280h^3-6b}{3}\rightarrow \infty
\end{equation}
as $h\rightarrow \infty$, where $v$ is the generator of $F_N^3$ in \eqref{eq8.20}. Therefore $exp(\mathbb{C}(N_x+N_y))F_N^{\bullet}$ is a $\mathbb{R}_{>0}\langle N_x+N_y\rangle$-nilpotent orbit.

Finally we check whether $exp(\mathbb{C}N_x)F_N^{\bullet}$ is horizontal. From \eqref{eq8.11}, \eqref{eq8.14}, \eqref{eq8.18}, we have:
\begin{equation}
    N_xF_N^3=\frac{35}{2}\alpha_0-15\alpha_1-20\alpha_2+\beta_1.
\end{equation}
Notice $exp(N)N_x(a_0\alpha_0+\beta_0)\in F_N^2$, and
\begin{equation}
\begin{split}
   Q(N_xF_N^3, exp(N)N_x(a_0\alpha_0+\beta_0)) &= Q(N_xexp(N)F^3, exp(N)N_xF^3)\\
    &=-Q(exp(N)F^3, N_xexp(N)N_xF^3)
\end{split}
\end{equation}
where
\begin{equation}
N_xexp(N)N_x=
\begin{pmatrix}
0 & 0 & 0 & 10 & 5 & 13\\
0 & 0 & 0 & 0 & 0 & -5\\
0 & 0 & 0 & 0 & 0 & -10\\
0 & 0 & 0 & 0 & 0 & 0\\
0 & 0 & 0 & 0 & 0 & 0\\
0 & 0 & 0 & 0 & 0 & 0
\end{pmatrix}, 
\end{equation}
Computation shows
\begin{equation}
 Q(N_xF_N^3, exp(N)N_x(a_0\alpha_0+\beta_0))=-2\neq 0
\end{equation}
This implies $N_xF_N^3\notin F_N^2$ and the Griffiths transversality fails. Therefore, $exp(\mathbb{C}N_x)F_N^{\bullet}$ is not horizontal, and so $exp(\sigma_{\mathbb{C}})F_N^{\bullet}$ is certainly not a nilpotent orbit.

\medskip

\printbibliography

%\printbibliography

\end{document}